\documentclass[11pt]{article}

\usepackage[T1]{fontenc}
\usepackage[margin=1in]{geometry}
\usepackage{amsmath,amssymb,amsfonts,amsthm,mathtools}
\usepackage{graphicx}
\usepackage{booktabs,tabularx}
\usepackage{enumitem}
\usepackage{algorithm}
\usepackage{algpseudocode}
\usepackage{natbib}
\usepackage{xcolor}
\usepackage{microtype}
\usepackage{placeins}
\usepackage{hyperref}

\graphicspath{{deep_dro_revised_assets/}}

\hypersetup{
    colorlinks=true,
    linkcolor=blue,
    citecolor=blue,
    urlcolor=blue
}

\newtheorem{theorem}{Theorem}[section]
\newtheorem{proposition}[theorem]{Proposition}
\newtheorem{lemma}[theorem]{Lemma}
\newtheorem{corollary}[theorem]{Corollary}
\newtheorem{assumption}[theorem]{Assumption}
\newtheorem{definition}[theorem]{Definition}
\newtheorem{remark}[theorem]{Remark}

\newcommand{\R}{\mathbb{R}}

\newcommand{\E}{\mathbb{E}}
\newcommand{\Prob}{\mathbb{P}}
\newcommand{\cP}{\mathcal{P}}
\newcommand{\cM}{\mathcal{M}}

\newcommand{\XiSet}{\Xi}
\newcommand{\1}{\mathbf{1}}
\newcommand{\dW}{d_{\mathrm W}}
\newcommand{\Phat}{\widehat P_N}
\newcommand{\Ptil}{\widetilde P_N}
\newcommand{\Pbar}{\bar P_N}
\newcommand{\epshat}{\widehat\varepsilon_N}
\newcommand{\epstilde}{\widetilde\varepsilon_N}
\newcommand{\epsmc}{\varepsilon_N^{\mathrm{mc}}}
\newcommand{\remp}{r_N^{\mathrm{emp}}}
\newcommand{\rpred}{r_N^{\mathrm{pred}}}
\newcommand{\rcert}{r_N^{\mathrm{cert}}}
\newcommand{\rdist}{r_N^{\mathrm{dist}}}
\newcommand{\inner}[2]{\left\langle #1,#2\right\rangle}
\newcommand{\norm}[1]{\left\lVert #1\right\rVert}
\DeclareMathOperator*{\argmin}{arg\,min}

\DeclareMathOperator{\Cl}{cl}

\title{Learning the Center and Radius of Wasserstein Ambiguity Sets for Data-Driven Decision Making}
\author{Junjie Guo}
\date{}

\begin{document}
\maketitle

\begin{abstract}
Wasserstein distributionally robust optimization (DRO) is commonly built around the empirical distribution, with the ambiguity radius selected from a concentration bound. Although this construction provides useful statistical guarantees, it can be conservative and does not fully exploit predictive information about the underlying distribution or the difficulty of a particular decision problem. We develop a more flexible framework in which a predictive model determines the nominal distribution and a separate model estimates a data-dependent radius. The key requirement is not that the ambiguity set be centered at the empirical distribution, but that it contain the unknown data-generating distribution with the desired probability. We establish finite-sample guarantees and asymptotic consistency for arbitrary learned centers, derive tractable reformulations for non-uniform discrete predictive distributions, separate predictive-model and scenario-discretization errors, and prove stability under simultaneous perturbations of the center and radius. We further characterize the oracle conditional-quantile radius as the smallest conditionally valid rule and introduce a split-conformal procedure for finite-sample marginal calibration. Experiments on newsvendor problems, synthetic portfolios, distribution shifts, and real financial data show that learned and calibrated ambiguity sets can improve reliability, but do not automatically yield smaller radii or better decisions. Overall, the proposed framework treats calibration as a practical mechanism for reliable decision making rather than a universal guarantee of improved optimization performance.
\end{abstract}

\noindent\textbf{Keywords:} distributionally robust optimization; Wasserstein distance; predictive distribution; learned radius; conformal calibration; finite-sample guarantee; portfolio optimization.

\section{Introduction}

Consider the stochastic optimization problem
\begin{equation}\label{eq:true-sp}
    J^\star = \inf_{x\in X} \E_P[h(x,\xi)],
\end{equation}
where $X\subseteq\R^n$ is a decision set, $\xi\in\XiSet\subseteq\R^m$ is a random vector, $P$ is its unknown data-generating distribution, and $h:X\times\XiSet\to\R$ is a loss. In practice $P$ is not directly observable. A common data-driven approach replaces $P$ by the empirical distribution
\begin{equation}\label{eq:empirical}
    \Phat = \frac1N\sum_{i=1}^N \delta_{\xi_i},
\end{equation}
where $\xi_1,\ldots,\xi_N\overset{\mathrm{i.i.d.}}{\sim}P$. This yields the sample-average approximation (SAA)
\begin{equation}\label{eq:saa}
    \widehat J_N^{\mathrm{saa}} = \inf_{x\in X} \frac1N\sum_{i=1}^N h(x,\xi_i).
\end{equation}
SAA is asymptotically consistent under standard regularity conditions \citep{shapiro2009lectures}, but it can overfit the empirical sample and display poor out-of-sample performance for small or moderate $N$.

Wasserstein DRO addresses this issue by optimizing against all distributions near a nominal distribution. In the foundational work of \citet{esfahani2018data}, the nominal distribution is the empirical law \eqref{eq:empirical}, and the ambiguity set is
\begin{equation}\label{eq:empirical-ball}
    B_{\varepsilon}(\Phat)=\{Q\in\cM(\XiSet):\dW(Q,\Phat)\le \varepsilon\}.
\end{equation}
The corresponding robust certificate is
\begin{equation}\label{eq:empirical-dro}
    \widehat J_N(\varepsilon)=\inf_{x\in X}\sup_{Q\in B_{\varepsilon}(\Phat)} \E_Q[h(x,\xi)].
\end{equation}
If $\varepsilon$ is chosen so that $P\in B_\varepsilon(\Phat)$ with probability at least $1-\beta$, then $\widehat J_N(\varepsilon)$ is an upper confidence bound on the out-of-sample cost of the robust optimizer. Under light-tail assumptions, \citet{esfahani2018data} use measure concentration to construct such an a priori radius.

This paper studies a modular alternative in which both the center and the radius are data-dependent. A predictive model first outputs a probability law $\widehat G_N$ using historical observations and, when available, side information. For tractable optimization, $\widehat G_N$ may be discretized into
\begin{equation}\label{eq:predictive-dist}
    \Ptil = \sum_{j=1}^{M_N}p_{j,N}\delta_{\zeta_{j,N}},\qquad p_{j,N}\ge0,\quad \sum_{j=1}^{M_N}p_{j,N}=1.
\end{equation}
A second model then predicts a radius
\begin{equation}\label{eq:learned-radius-intro}
    \epshat = g_\theta(Z_N),
\end{equation}
where $Z_N$ may contain sample size, dimension, predictive diagnostics, volatility or regime features, and uncertainty estimates. The key requirement is
\begin{equation}\label{eq:coverage-intro}
    P^N\left\{\dW(P,\Pbar)\le \epshat\right\}\ge1-\beta,
\end{equation}
where $\Pbar$ denotes the actual nominal law used by the optimizer. Thus, the empirical law is not logically required for the standard certificate argument; what is required is calibrated coverage of the unknown law by the random ambiguity set.

The statement above is intentionally modest: it is a coverage principle, not by itself a complete learning theory. The substantive questions are how predictive-model error, finite scenario generation, radius estimation, and calibration combine; whether the resulting robust value is stable; and whether a learned radius is efficient among valid rules. These are the points addressed by the new quantitative results in this paper.

\paragraph{Contributions.} The paper makes the following contributions.
\begin{enumerate}[leftmargin=2em]
    \item It gives a unified formulation for arbitrary random nominal distributions and records the standard strong dual and finite convex reformulations for a non-uniform discrete center. This extension is technically straightforward but is needed for predictive scenario distributions.
    \item It proves finite-sample validity and asymptotic consistency for learned centers, together with exact and approximate radius dominance and Lipschitz regret and certificate bounds.
    \item It separates predictive-model error from scenario-discretization error and shows how the two errors add in the radius. It also establishes a joint stability bound for the robust value under simultaneous perturbations of the center and radius.
    \item It studies learned random radii. The oracle conditional quantile is shown to be the pointwise smallest conditionally valid radius rule, approximate radius predictors are corrected by safety margins, and a finite-sample-correct split-conformal procedure is derived.
    \item It distinguishes universal distributional coverage from task-specific certificate coverage and provides both high-probability and expected-regret guarantees when radius undercoverage is nonzero.
    \item It develops a mean--CVaR portfolio reformulation and implements a leakage-free evaluation with nested fitting, calibration, and testing; paired inference; distribution-shift stress tests; and transaction costs. The experiments find reliable calibration in controlled settings but no uniform radius-efficiency or decision-value advantage.
\end{enumerate}

\section{Related Literature}

DRO replaces a single estimated distribution by an ambiguity set and has been developed under moment, divergence, and transport-based descriptions of distributional uncertainty; representative foundations include \citet{bental2013robust}, \citet{delage2010distributionally}, \citet{wiesemann2014distributionally}, and the review of \citet{rahimian2022frameworks}. Wasserstein ambiguity sets are especially attractive because they retain the geometry of the sample space and often admit strong duality and tractable reformulations \citep{esfahani2018data,gao2023distributionally,blanchet2019quantifying,kuhn2019wasserstein}. Their statistical behavior is tied to convergence of empirical measures in Wasserstein distance \citep{fournier2015rate,weed2019sharp}, while recent work gives broader generalization and regularization interpretations \citep{wu2025generalization}.

The ambiguity-set size is critical. Concentration inequalities provide finite-sample radii but can be conservative, while robust Wasserstein profile inference, statistically calibrated empirical optimization, bootstrap procedures, and validation-based rules offer more adaptive alternatives \citep{blanchet2019robust,gotoh2021calibration,bertsimas2022bootstrap}. The present paper does not claim that data-adaptive ambiguity sets are new. Its focus is the joint decomposition of center learning, scenario discretization, and radius calibration, together with stability and efficiency results for the resulting random ball.

The learned-center viewpoint is related to predictive and contextual optimization. Data-driven robust optimization constructs uncertainty sets from statistical tests \citep{bertsimas2018datadriven}; predictive-to-prescriptive methods use covariates to form decisions \citep{bertsimas2020predictive}; and residual-based contextual DRO centers ambiguity sets around regression predictions and studies asymptotic and finite-sample properties \citep{kannan2024residuals}. Our formulation is deliberately model-agnostic: the predictive center may be parametric, nonparametric, or generative, while the radius is treated as a separate calibrated statistical object. In portfolio optimization, optimal-transport robustness has also been combined directly with side information and conditional estimation \citep{nguyen2024robustifying}.

Finally, the calibration component is connected to conformal prediction. Split conformal and conformalized quantile regression provide finite-sample marginal coverage under exchangeability \citep{romano2019conformalized}. Weighted and adaptive variants address distribution drift and sequential data, although their guarantees differ from ordinary exchangeable split conformal \citep{barber2023conformal,gibbs2021adaptive}. This distinction is essential in rolling financial experiments, where naive random splitting would create leakage and invalidate the nominal guarantee.

\section{Notation and Preliminaries}

Let $\cM(\XiSet)$ be the set of probability distributions supported on $\XiSet$ with finite first moment. For $Q_1,Q_2\in\cM(\XiSet)$, the 1-Wasserstein distance induced by a norm $\norm{\cdot}$ is
\begin{equation}\label{eq:wasserstein-primal}
    \dW(Q_1,Q_2)=\inf_{\Pi\in\Gamma(Q_1,Q_2)}\int_{\XiSet\times\XiSet}\norm{\xi-\xi'}\,\Pi(d\xi,d\xi'),
\end{equation}
where $\Gamma(Q_1,Q_2)$ is the set of couplings with marginals $Q_1$ and $Q_2$. Its Kantorovich--Rubinstein dual representation \citep{kantorovich1958space,villani2009optimal} is
\begin{equation}\label{eq:KR}
    \dW(Q_1,Q_2)=\sup_{f\in\mathrm{Lip}_1}\left\{\E_{Q_1}[f(\xi)]-\E_{Q_2}[f(\xi)]\right\},
\end{equation}
where $\mathrm{Lip}_1$ is the class of functions satisfying $|f(\xi)-f(\xi')|\le\norm{\xi-\xi'}$.

The dual norm is $\norm{z}_*=\sup_{\norm{\xi}\le1}\inner{z}{\xi}$. For a set $\XiSet$, its characteristic function is $\chi_{\XiSet}(\xi)=0$ if $\xi\in\XiSet$ and $+\infty$ otherwise. Its support function is
\begin{equation}
    \sigma_{\XiSet}(z)=\sup_{\xi\in\XiSet}\inner{z}{\xi}.
\end{equation}
For a proper function $f$, its convex conjugate is
\begin{equation}
    f^*(z)=\sup_{\xi}\{\inner{z}{\xi}-f(\xi)\}.
\end{equation}

Throughout, randomness of $\Phat$, $\Ptil$, $\Pbar$, $\epshat$, and all data-driven decisions is with respect to the product law $P^N$ unless stated otherwise. If the predictive model or scenario generator uses additional independent randomization, the underlying probability space is enlarged to include that randomness, and the same notation is retained for simplicity.

\section{Classical Empirical Wasserstein DRO}

This section recalls the empirical construction of \citet{esfahani2018data}. The empirical Wasserstein DRO problem is \eqref{eq:empirical-dro}. If $\widehat x_N(\varepsilon)$ is an optimizer and $\widehat J_N(\varepsilon)$ is the optimal value, the desired finite-sample guarantee is
\begin{equation}\label{eq:desired-guarantee}
    P^N\left\{\E_P[h(\widehat x_N(\varepsilon),\xi)]\le \widehat J_N(\varepsilon)\right\}\ge 1-\beta.
\end{equation}

\subsection{Measure-concentration radius}

Assume the following light-tail condition.

\begin{assumption}[Light tails]\label{ass:light-tail}
There exists $a>1$ such that
\begin{equation}
    A:=\E_P\left[\exp(\norm{\xi}^a)\right]<\infty.
\end{equation}
\end{assumption}

A measure concentration result due to \citet{fournier2015rate} implies that, for constants $c_1,c_2>0$ depending only on $a,A,m$,
\begin{equation}\label{eq:mc-ineq}
P^N\{\dW(P,\Phat)\ge\varepsilon\}\le
\begin{cases}
    c_1\exp(-c_2N\varepsilon^{\max\{m,2\}}), & \varepsilon\le1,\\
    c_1\exp(-c_2N\varepsilon^a), & \varepsilon>1.
\end{cases}
\end{equation}
Equating the right side to $\beta$ gives the a priori radius
\begin{equation}\label{eq:eps-mc}
\epsmc(\beta)=
\begin{cases}
\left(\dfrac{\log(c_1\beta^{-1})}{c_2N}\right)^{1/\max\{m,2\}}, & N\ge \dfrac{\log(c_1\beta^{-1})}{c_2},\\[1.3em]
\left(\dfrac{\log(c_1\beta^{-1})}{c_2N}\right)^{1/a}, & N< \dfrac{\log(c_1\beta^{-1})}{c_2}.
\end{cases}
\end{equation}
Then
\begin{equation}
    P^N\{P\in B_{\epsmc(\beta)}(\Phat)\}\ge1-\beta.
\end{equation}

\begin{theorem}[Empirical finite-sample guarantee]\label{thm:emp-guarantee}
Let $\widehat J_N=\widehat J_N(\epsmc(\beta))$ and let $\widehat x_N$ be an optimizer of \eqref{eq:empirical-dro} with radius $\epsmc(\beta)$. Under Assumption \ref{ass:light-tail},
\begin{equation}
    P^N\left\{\E_P[h(\widehat x_N,\xi)]\le \widehat J_N\right\}\ge1-\beta.
\end{equation}
\end{theorem}

\begin{proof}
By the construction of $\epsmc(\beta)$, $P\in B_{\epsmc(\beta)}(\Phat)$ with probability at least $1-\beta$. On this event,
\begin{align*}
    \E_P[h(\widehat x_N,\xi)]
    &\le \sup_{Q\in B_{\epsmc(\beta)}(\Phat)}\E_Q[h(\widehat x_N,\xi)]\\
    &= \widehat J_N,
\end{align*}
where the equality follows from optimality of $\widehat x_N$. This proves the claim.
\end{proof}

\subsection{Limitations of the a priori radius}

The radius \eqref{eq:eps-mc} is universal and safe, but it is not necessarily tight. It depends on $N$, $\beta$, $m$, and tail constants, but not on the realized geometry or difficulty of the observed dataset. In high dimensions, the rate in \eqref{eq:eps-mc} can be pessimistic. This motivates data-adaptive radii, including bootstrap, cross-validation, and the learned-radius approach developed later.

\section{Worst-Case Expectation Duality for a General Discrete Nominal Distribution}\label{sec:dual}

The results of \citet{esfahani2018data} are stated for the empirical nominal law $\Phat$, whose weights are all $1/N$. We need the corresponding formulas for an arbitrary discrete nominal law
\begin{equation}\label{eq:general-nominal}
    \Pbar=\sum_{j=1}^{M}p_j\delta_{\zeta_j},\qquad p_j>0,\quad \sum_{j=1}^M p_j=1.
\end{equation}
This includes $\Phat$ as the special case $M=N$, $p_j=1/N$, $\zeta_j=\xi_j$, and it includes predictive distributions of the form \eqref{eq:predictive-dist}.

Let $\ell:\XiSet\to\R$ be a decision-independent loss. Consider
\begin{equation}\label{eq:wce-general}
    \Phi_{\varepsilon}(\ell;\Pbar)=\sup_{Q\in B_\varepsilon(\Pbar)}\E_Q[\ell(\xi)].
\end{equation}

\begin{theorem}[Dual worst-case expectation]\label{thm:dual-general}
Suppose strong duality holds for the moment problem in \eqref{eq:wce-general}; in particular, this is true under the standard upper-semicontinuity and linear-growth conditions used in Wasserstein DRO. Then
\begin{equation}\label{eq:dual-general}
    \Phi_{\varepsilon}(\ell;\Pbar)
    =\inf_{\lambda\ge0}\left\{\lambda\varepsilon+\sum_{j=1}^Mp_j\sup_{\xi\in\XiSet}\left(\ell(\xi)-\lambda\norm{\xi-\zeta_j}\right)\right\}.
\end{equation}
\end{theorem}

\begin{proof}
Using the definition of the Wasserstein ball, write
\begin{align*}
\Phi_{\varepsilon}(\ell;\Pbar)
=\sup_{Q,\Pi}\ &\int_{\XiSet}\ell(\xi)Q(d\xi)\\
\text{s.t. }&\int_{\XiSet\times\XiSet}\norm{\xi-\xi'}\,\Pi(d\xi,d\xi')\le\varepsilon,\\
&\Pi\text{ has marginals }Q\text{ and }\Pbar.
\end{align*}
Since $\Pbar=\sum_jp_j\delta_{\zeta_j}$ is discrete, every coupling $\Pi$ with second marginal $\Pbar$ can be disintegrated as
\begin{equation}
    \Pi(d\xi,d\xi')=\sum_{j=1}^Mp_jQ_j(d\xi)\delta_{\zeta_j}(d\xi'),
\end{equation}
where $Q_j\in\cM(\XiSet)$ are conditional distributions. The first marginal is $Q=\sum_jp_jQ_j$. Hence
\begin{align*}
\Phi_{\varepsilon}(\ell;\Pbar)
=\sup_{Q_j\in\cM(\XiSet)}\ &\sum_{j=1}^Mp_j\int\ell(\xi)Q_j(d\xi)\\
\text{s.t. }&\sum_{j=1}^Mp_j\int\norm{\xi-\zeta_j}Q_j(d\xi)\le\varepsilon.
\end{align*}
Dualizing the transportation-budget constraint with multiplier $\lambda\ge0$ gives
\begin{align*}
\Phi_{\varepsilon}(\ell;\Pbar)
&=\inf_{\lambda\ge0}\left\{\lambda\varepsilon+
\sup_{Q_j}\sum_{j=1}^Mp_j\int\left(\ell(\xi)-\lambda\norm{\xi-\zeta_j}\right)Q_j(d\xi)\right\}\\
&=\inf_{\lambda\ge0}\left\{\lambda\varepsilon+
\sum_{j=1}^Mp_j\sup_{\xi\in\XiSet}\left(\ell(\xi)-\lambda\norm{\xi-\zeta_j}\right)\right\},
\end{align*}
where the last equality uses the fact that each $Q_j$ may concentrate as a Dirac mass at a maximizer or along a maximizing sequence. This proves \eqref{eq:dual-general}.
\end{proof}

\subsection{Finite convex reformulation}

We now state the finite convex reformulation. It is the non-uniform-weight version of the convex reduction in \citet{esfahani2018data}.

\begin{assumption}[Concave-piece loss]\label{ass:concave-piece}
The uncertainty set $\XiSet\subseteq\R^m$ is closed and convex. The loss admits the representation
\begin{equation}\label{eq:max-concave}
    \ell(\xi)=\max_{k\le K}\ell_k(\xi),
\end{equation}
where $-\ell_k$ is proper, closed, and convex for every $k$.
\end{assumption}

\begin{theorem}[Finite convex reformulation]\label{thm:convex-reform-general}
Under Assumption \ref{ass:concave-piece} and the strong-duality conditions of Theorem \ref{thm:dual-general},
\begin{equation}\label{eq:convex-reform-general}
\Phi_{\varepsilon}(\ell;\Pbar)=
\begin{aligned}[t]
\inf_{\lambda,s_j,z_{jk},\nu_{jk}}\quad
&\lambda\varepsilon+\sum_{j=1}^Mp_js_j\\
\mathrm{s.t.}\quad
&[-\ell_k]^*(z_{jk}-\nu_{jk})+\sigma_{\XiSet}(\nu_{jk})-\inner{z_{jk}}{\zeta_j}\le s_j,\quad \forall j,k,\\
&\norm{z_{jk}}_*\le\lambda,\quad \forall j,k.
\end{aligned}
\end{equation}
\end{theorem}

\begin{proof}
Starting from \eqref{eq:dual-general}, introduce epigraph variables $s_j$:
\begin{equation}
\inf_{\lambda\ge0,s_j}\left\{\lambda\varepsilon+\sum_jp_js_j:
\sup_{\xi\in\XiSet}\left(\ell(\xi)-\lambda\norm{\xi-\zeta_j}\right)\le s_j,\ \forall j\right\}.
\end{equation}
Using $\ell=\max_k\ell_k$, the constraints are equivalent to
\begin{equation}\label{eq:semiinf}
    \sup_{\xi\in\XiSet}\left(\ell_k(\xi)-\lambda\norm{\xi-\zeta_j}\right)\le s_j,
    \qquad \forall j,k.
\end{equation}
By the definition of the dual norm,
\begin{equation}
    \lambda\norm{\xi-\zeta_j}=\sup_{\norm{z}_{*}\le\lambda}\inner{z}{\xi-\zeta_j}.
\end{equation}
Therefore, under the concavity/convexity conditions and the minimax theorem,
\begin{align*}
\sup_{\xi\in\XiSet}\left(\ell_k(\xi)-\lambda\norm{\xi-\zeta_j}\right)
&=\inf_{\norm{z_{jk}}_*
\le\lambda}\sup_{\xi\in\XiSet}\left(\ell_k(\xi)-\inner{z_{jk}}{\xi-\zeta_j}\right)\\
&=\inf_{\norm{z_{jk}}_*
\le\lambda}\left\{[-\ell_k+\chi_{\XiSet}]^*(-z_{jk})+\inner{z_{jk}}{\zeta_j}\right\}.
\end{align*}
Replacing $z_{jk}$ by $-z_{jk}$ and using inf-convolution of conjugates,
\begin{equation}
[-\ell_k+\chi_{\XiSet}]^*(z_{jk})
=\Cl\inf_{\nu_{jk}}\left\{[-\ell_k]^*(z_{jk}-\nu_{jk})+\sigma_{\XiSet}(\nu_{jk})\right\}.
\end{equation}
Substituting this expression into the epigraph constraints yields \eqref{eq:convex-reform-general}. The closure does not change the resulting inequality constraints under standard lower-semicontinuity conventions.
\end{proof}

\begin{corollary}[Piecewise affine loss over full space]\label{cor:pwa-fullspace}
Suppose $\XiSet=\R^m$ and
\begin{equation}
    \ell(\xi)=\max_{k\le K}\{\inner{a_k}{\xi}+b_k\}.
\end{equation}
Then
\begin{equation}\label{eq:pwa-fullspace}
\Phi_\varepsilon(\ell;\Pbar)=
\begin{aligned}[t]
\inf_{\lambda,s_j}\quad
&\lambda\varepsilon+\sum_{j=1}^Mp_js_j\\
\mathrm{s.t.}\quad
&\inner{a_k}{\zeta_j}+b_k\le s_j,\quad \forall j,k,\\
&\norm{a_k}_*\le\lambda,
\quad \forall k.
\end{aligned}
\end{equation}
Equivalently,
\begin{equation}
    \Phi_\varepsilon(\ell;\Pbar)=\sum_{j=1}^Mp_j\ell(\zeta_j)+\varepsilon\max_{k\le K}\norm{a_k}_*.
\end{equation}
\end{corollary}

\begin{proof}
For $\XiSet=\R^m$, the support function is $\sigma_{\R^m}(\nu)=\chi_{\{0\}}(\nu)$. For affine pieces, $[-\ell_k]^*(z)$ is finite only when $z=-a_k$. Substituting these identities into Theorem \ref{thm:convex-reform-general} gives \eqref{eq:pwa-fullspace}. At the optimum, $s_j=\max_k\{\inner{a_k}{\zeta_j}+b_k\}$ and $\lambda=\max_k\norm{a_k}_*$.
\end{proof}

\section{Predictive Nominal Distributions}\label{sec:predictive-dist}

We now replace the empirical center by a learned predictive distribution $\Ptil$ of the form \eqref{eq:predictive-dist}. Define the predictive Wasserstein ambiguity set
\begin{equation}
    \widetilde{\cP}_N(\varepsilon)=B_\varepsilon(\Ptil),
\end{equation}
and the predictive DRO value
\begin{equation}\label{eq:pred-dro}
    \widetilde J_N(\varepsilon)=\inf_{x\in X}\sup_{Q\in B_\varepsilon(\Ptil)}\E_Q[h(x,\xi)].
\end{equation}
Let $\widetilde x_N(\varepsilon)$ be an optimizer when it exists.

\subsection{Finite-sample guarantee}

\begin{theorem}[Predictive finite-sample guarantee]\label{thm:pred-finite}
Suppose that, for some radius $\epstilde(\beta)$,
\begin{equation}\label{eq:pred-coverage}
    P^N\left\{\dW(P,\Ptil)\le \epstilde(\beta)\right\}\ge1-\beta.
\end{equation}
Let $\widetilde J_N=\widetilde J_N(\epstilde(\beta))$ and let $\widetilde x_N=\widetilde x_N(\epstilde(\beta))$. Then
\begin{equation}
    P^N\left\{\E_P[h(\widetilde x_N,\xi)]\le \widetilde J_N\right\}\ge1-\beta.
\end{equation}
\end{theorem}

\begin{proof}
Define $\mathcal E_N=\{\dW(P,\Ptil)\le\epstilde(\beta)\}$. On $\mathcal E_N$, $P\in B_{\epstilde(\beta)}(\Ptil)$. Hence
\begin{equation*}
\E_P[h(\widetilde x_N,\xi)]\le
\sup_{Q\in B_{\epstilde(\beta)}(\Ptil)}\E_Q[h(\widetilde x_N,\xi)]=\widetilde J_N.
\end{equation*}
The event has probability at least $1-\beta$ by assumption.
\end{proof}

\begin{remark}
Theorem \ref{thm:pred-finite} shows that the empirical distribution is not essential. Any random nominal distribution can be used if one can calibrate a high-probability Wasserstein error bound for it.
\end{remark}

\subsection{Predictive radius dominance over empirical radius}

Define the empirical and predictive Wasserstein errors
\begin{equation}
    R_N^{\mathrm{emp}}=\dW(P,\Phat),\qquad
    R_N^{\mathrm{pred}}=\dW(P,\Ptil).
\end{equation}

\begin{theorem}[Coverage dominance]\label{thm:coverage-dom}
Assume
\begin{equation}\label{eq:pred-closer-a-s}
    R_N^{\mathrm{pred}}\le R_N^{\mathrm{emp}}
\end{equation}
for every dataset realization. Then, for every $\varepsilon\ge0$,
\begin{equation}
    \{P\in B_\varepsilon(\Phat)\}\subseteq \{P\in B_\varepsilon(\Ptil)\},
\end{equation}
and therefore
\begin{equation}
    P^N\{P\in B_\varepsilon(\Ptil)\}\ge P^N\{P\in B_\varepsilon(\Phat)\}.
\end{equation}
\end{theorem}

\begin{proof}
If $P\in B_\varepsilon(\Phat)$, then $R_N^{\mathrm{emp}}\le\varepsilon$. By \eqref{eq:pred-closer-a-s}, $R_N^{\mathrm{pred}}\le R_N^{\mathrm{emp}}\le\varepsilon$, hence $P\in B_\varepsilon(\Ptil)$. Taking probabilities proves the second claim.
\end{proof}

Define minimal confidence radii
\begin{align}
\remp(\beta)&=\inf\left\{r\ge0:P^N(R_N^{\mathrm{emp}}\le r)\ge1-\beta\right\},\label{eq:remp}\\
\rpred(\beta)&=\inf\left\{r\ge0:P^N(R_N^{\mathrm{pred}}\le r)\ge1-\beta\right\}.\label{eq:rpred}
\end{align}

\begin{corollary}[Confidence-radius dominance]\label{cor:radius-dom}
Under \eqref{eq:pred-closer-a-s},
\begin{equation}
    \rpred(\beta)\le \remp(\beta).
\end{equation}
\end{corollary}

\begin{proof}
By Theorem \ref{thm:coverage-dom}, for every $r\ge0$,
\begin{equation}
    P^N(R_N^{\mathrm{pred}}\le r)\ge P^N(R_N^{\mathrm{emp}}\le r).
\end{equation}
Thus every radius that attains empirical confidence $1-\beta$ also attains predictive confidence $1-\beta$. Taking infima gives the result.
\end{proof}

\begin{theorem}[High-probability dominance]\label{thm:hp-radius-dom}
Let $\alpha,\beta\in(0,1)$ satisfy $\alpha+\beta<1$. Suppose
\begin{equation}
    P^N\{R_N^{\mathrm{pred}}\le R_N^{\mathrm{emp}}\}\ge1-\alpha.
\end{equation}
Then, for every $r\ge0$,
\begin{equation}
    P^N(R_N^{\mathrm{pred}}\le r)\ge P^N(R_N^{\mathrm{emp}}\le r)-\alpha.
\end{equation}
Consequently,
\begin{equation}
    \rpred(\beta+\alpha)\le \remp(\beta).
\end{equation}
\end{theorem}

\begin{proof}
Let $A_N=\{R_N^{\mathrm{pred}}\le R_N^{\mathrm{emp}}\}$. Then
\begin{equation}
    A_N\cap\{R_N^{\mathrm{emp}}\le r\}\subseteq\{R_N^{\mathrm{pred}}\le r\}.
\end{equation}
Therefore,
\begin{align*}
P^N(R_N^{\mathrm{pred}}\le r)
&\ge P^N(A_N\cap\{R_N^{\mathrm{emp}}\le r\})\\
&\ge P^N(R_N^{\mathrm{emp}}\le r)-P^N(A_N^c)\\
&\ge P^N(R_N^{\mathrm{emp}}\le r)-\alpha.
\end{align*}
If $r=\remp(\beta)$ up to an arbitrarily small slack, the right side is at least $1-\beta-\alpha$. Taking the infimum yields $\rpred(\beta+\alpha)\le\remp(\beta)$.
\end{proof}

\begin{theorem}[Approximate confidence-radius dominance]\label{thm:approx-radius-dom}
Suppose there is a deterministic $\Delta_N\ge0$ such that
\begin{equation}
    R_N^{\mathrm{pred}}\le R_N^{\mathrm{emp}}+\Delta_N
    \qquad\text{almost surely.}
\end{equation}
Then, for every $\beta\in(0,1)$,
\begin{equation}
    \rpred(\beta)\le \remp(\beta)+\Delta_N.
\end{equation}
More generally, if $\alpha+\beta<1$ and
\begin{equation}
P^N\{R_N^{\mathrm{pred}}\le R_N^{\mathrm{emp}}+\Delta_N\}\ge1-\alpha,
\end{equation}
then
\begin{equation}
    \rpred(\beta+\alpha)\le \remp(\beta)+\Delta_N.
\end{equation}
\end{theorem}

\begin{proof}
For any $r\ge0$, the almost-sure inequality implies
\begin{equation}
\{R_N^{\mathrm{emp}}\le r\}\subseteq
\{R_N^{\mathrm{pred}}\le r+\Delta_N\}.
\end{equation}
Thus every empirical radius $r$ with coverage at least $1-\beta$ yields a predictive radius $r+\Delta_N$ with the same coverage. Taking infima proves the first claim. For the high-probability statement, intersect $\{R_N^{\mathrm{emp}}\le r\}$ with the event on which the approximate inequality holds and use the same union-bound argument as in Theorem \ref{thm:hp-radius-dom}.
\end{proof}

\subsection{Perfect and approximately perfect predictive models}

\begin{definition}[Perfect predictive model]
The predictive model is perfect on a dataset if
\begin{equation}
    \Ptil=P.
\end{equation}
Define $\pi_N=P^N\{\Ptil=P\}$.
\end{definition}

\begin{proposition}[Perfect model reduces DRO to the true problem]\label{prop:perfect}
If $\Ptil=P$, then with radius $0$,
\begin{equation}
    \widetilde J_N(0)=J^\star.
\end{equation}
Moreover,
\begin{equation}
    P^N\{\widetilde J_N(0)=J^\star\}\ge\pi_N.
\end{equation}
\end{proposition}

\begin{proof}
If $\Ptil=P$, then $B_0(\Ptil)=\{P\}$. Hence
\begin{equation*}
\widetilde J_N(0)=\inf_{x\in X}\sup_{Q\in\{P\}}\E_Q[h(x,\xi)]
=\inf_{x\in X}\E_P[h(x,\xi)]=J^\star.
\end{equation*}
Taking probabilities gives the second statement.
\end{proof}

\begin{remark}
If $P$ is continuous and $\Ptil$ is finite discrete, exact equality usually has probability zero. The practically relevant version is approximate perfection: $\dW(P,\Ptil)\le r_N$ with $r_N\to0$.
\end{remark}

\subsection{Predictive asymptotic consistency}

\begin{assumption}[Regularity for consistency]\label{ass:regularity}
The loss $h(x,\xi)$ is upper semicontinuous in $\xi$ for every $x\in X$, and there exists $L_0\ge0$ such that
\begin{equation}
    |h(x,\xi)|\le L_0(1+\norm{\xi})
    \qquad \forall x\in X,\ \xi\in\XiSet.
\end{equation}
For optimizer convergence, assume additionally that $X$ is closed and $h(x,\xi)$ is lower semicontinuous in $x$ for every $\xi$.
\end{assumption}

\begin{theorem}[Predictive asymptotic consistency]\label{thm:pred-consistency}
Let $\varepsilon_N\to0$ and suppose
\begin{equation}\label{eq:eventual-pred-coverage}
    \dW(P,\Ptil)\le\varepsilon_N
\end{equation}
for all sufficiently large $N$, almost surely. Under Assumption \ref{ass:regularity},
\begin{equation}
    \widetilde J_N(\varepsilon_N)\to J^\star
    \qquad\text{almost surely.}
\end{equation}
If the additional optimizer-regularity conditions in Assumption \ref{ass:regularity} hold, then every accumulation point of $\{\widetilde x_N(\varepsilon_N)\}$ is an optimizer of \eqref{eq:true-sp}.
\end{theorem}

\begin{proof}
The eventual coverage condition implies
\begin{equation}\label{eq:lower-sandwich}
J^\star\le \E_P[h(\widetilde x_N,\xi)]\le \widetilde J_N(\varepsilon_N)
\end{equation}
for all sufficiently large $N$ almost surely. It remains to prove $\limsup_N\widetilde J_N(\varepsilon_N)\le J^\star$.

Fix $\delta>0$ and choose $x_\delta\in X$ such that
\begin{equation}
    \E_P[h(x_\delta,\xi)]\le J^\star+\delta.
\end{equation}
For each $N$, let $Q_N\in B_{\varepsilon_N}(\Ptil)$ be $\delta$-suboptimal for the inner supremum at $x_\delta$, i.e.,
\begin{equation}
\sup_{Q\in B_{\varepsilon_N}(\Ptil)}\E_Q[h(x_\delta,\xi)]\le \E_{Q_N}[h(x_\delta,\xi)]+\delta.
\end{equation}
By the triangle inequality and \eqref{eq:eventual-pred-coverage},
\begin{equation}
    \dW(Q_N,P)\le \dW(Q_N,\Ptil)+\dW(\Ptil,P)\le2\varepsilon_N\to0.
\end{equation}
Under upper semicontinuity and linear growth, convergence in Wasserstein distance implies the portmanteau-type upper semicontinuity
\begin{equation}
    \limsup_{N\to\infty}\E_{Q_N}[h(x_\delta,\xi)]\le \E_P[h(x_\delta,\xi)].
\end{equation}
Therefore,
\begin{align*}
\limsup_{N\to\infty}\widetilde J_N(\varepsilon_N)
&\le \limsup_{N\to\infty}\sup_{Q\in B_{\varepsilon_N}(\Ptil)}\E_Q[h(x_\delta,\xi)]\\
&\le \E_P[h(x_\delta,\xi)]+\delta\\
&\le J^\star+2\delta.
\end{align*}
Letting $\delta\downarrow0$ yields $\limsup_N\widetilde J_N(\varepsilon_N)\le J^\star$. Together with \eqref{eq:lower-sandwich}, this proves value convergence.

For optimizer convergence, the following argument is a direct variational limit argument in the spirit of \citet{rockafellar1998variational}. Let $x^\star$ be an accumulation point and pass to a convergent subsequence $\widetilde x_{N_k}\to x^\star$. Closedness gives $x^\star\in X$. The common lower envelope $-L_0(1+\norm{\xi})$ is $P$-integrable because $P\in\cM(\XiSet)$. Applying Fatou's lemma after shifting by this envelope, and using lower semicontinuity in $x$, gives
\begin{align*}
J^\star
&\le \E_P[h(x^\star,\xi)]\\
&\le \E_P\left[\liminf_{k\to\infty}h(\widetilde x_{N_k},\xi)\right]\\
&\le \liminf_{k\to\infty}\E_P[h(\widetilde x_{N_k},\xi)]\\
&\le \liminf_{k\to\infty}\widetilde J_{N_k}(\varepsilon_{N_k})=J^\star.
\end{align*}
Thus $\E_P[h(x^\star,\xi)]=J^\star$.
\end{proof}

\subsection{Lipschitz regret and certificate bounds}

\begin{assumption}[Uniform Lipschitz loss]\label{ass:lipschitz}
There exists $L>0$ such that for all $x\in X$ and all $\xi,\xi'\in\XiSet$,
\begin{equation}
    |h(x,\xi)-h(x,\xi')|\le L\norm{\xi-\xi'}.
\end{equation}
\end{assumption}

\begin{theorem}[Predictive Lipschitz regret bound]\label{thm:pred-regret}
Let $r_N$ be any bound satisfying $\dW(P,\Ptil)\le r_N$.  Fix $\varepsilon_N\ge0$ and write
\[
    \widetilde x_N=\widetilde x_N(\varepsilon_N).
\]
Under Assumption \ref{ass:lipschitz}, on the coverage event $\{\dW(P,\Ptil)\le\varepsilon_N\}$,
\begin{equation}\label{eq:general-regret-pred}
    0\le \E_P[h(\widetilde x_N,\xi)]-J^\star
    \le L(\varepsilon_N+r_N).
\end{equation}
In particular, the event is guaranteed whenever $r_N\le\varepsilon_N$, and choosing $\varepsilon_N=r_N$ gives
\begin{equation}\label{eq:2L-regret-pred}
    0\le \E_P[h(\widetilde x_N,\xi)]-J^\star\le2Lr_N.
\end{equation}
\end{theorem}

\begin{proof}
The lower bound follows from the definition of $J^\star$. On the coverage event,
\begin{equation}
    \E_P[h(\widetilde x_N,\xi)]\le \widetilde J_N(\varepsilon_N).
\end{equation}
Fix $\delta>0$ and choose a $\delta$-optimal true decision $x_\delta\in X$ such that $\E_P[h(x_\delta,\xi)]\le J^\star+\delta$. By optimality of $\widetilde x_N$,
\begin{equation}
    \widetilde J_N(\varepsilon_N)
    \le \sup_{Q\in B_{\varepsilon_N}(\Ptil)}\E_Q[h(x_\delta,\xi)].
\end{equation}
For any such $Q$,
\begin{equation}
    \dW(Q,P)\le \dW(Q,\Ptil)+\dW(\Ptil,P)
    \le\varepsilon_N+r_N.
\end{equation}
By Kantorovich--Rubinstein duality and Lipschitzness,
\begin{equation}
    \E_Q[h(x_\delta,\xi)]\le\E_P[h(x_\delta,\xi)]+L\dW(Q,P)
    \le J^\star+\delta+L(\varepsilon_N+r_N).
\end{equation}
Take the supremum over $Q$ and then let $\delta\downarrow0$.
\end{proof}

\begin{corollary}[Improved confidence-level regret bound under radius dominance]\label{cor:improved-regret}
If $\rpred(\beta)\le\remp(\beta)$, then under Assumption \ref{ass:lipschitz}, using the respective $(1-\beta)$ confidence radii gives the predictive high-probability regret bound
\begin{equation}
    2L\rpred(\beta),
\end{equation}
which is no larger than the empirical-center bound $2L\remp(\beta)$. Each bound applies on its corresponding coverage event.
\end{corollary}

\begin{theorem}[Certificate gap]\label{thm:cert-gap}
Under Assumption \ref{ass:lipschitz}, if $\dW(P,\Ptil)\le r_N$, then
\begin{equation}
    \widetilde J_N(\varepsilon_N)\le J^\star+L(\varepsilon_N+r_N).
\end{equation}
If $r_N\le\varepsilon_N$, then $J^\star\le \widetilde J_N(\varepsilon_N)$, and hence
\begin{equation}
    0\le \widetilde J_N(\varepsilon_N)-J^\star\le L(\varepsilon_N+r_N).
\end{equation}
In particular, if $\varepsilon_N=r_N$, then
\begin{equation}
    0\le \widetilde J_N(r_N)-J^\star\le2Lr_N.
\end{equation}
\end{theorem}

\begin{proof}
The upper bound follows by evaluating the robust objective at a $\delta$-optimal true decision as in the proof of Theorem \ref{thm:pred-regret} and then letting $\delta\downarrow0$. If $r_N\le\varepsilon_N$, then $P\in B_{\varepsilon_N}(\Ptil)$, so for every $x$,
\begin{equation}
    \sup_{Q\in B_{\varepsilon_N}(\Ptil)}\E_Q[h(x,\xi)]\ge \E_P[h(x,\xi)].
\end{equation}
Taking the infimum over $x$ gives $\widetilde J_N(\varepsilon_N)\ge J^\star$.
\end{proof}

\subsection{Stability with respect to the center and radius}

For a nominal law $\mu\in\cM(\XiSet)$ and radius $r\ge0$, define
\begin{align}
\phi_x(\mu,r)&=\sup_{Q:\dW(Q,\mu)\le r}\E_Q[h(x,\xi)],\\
V(\mu,r)&=\inf_{x\in X}\phi_x(\mu,r).
\end{align}

\begin{lemma}[Lipschitz dependence on the radius]\label{lem:radius-lipschitz}
Under Assumption \ref{ass:lipschitz}, for every $x\in X$, $\mu\in\cM(\XiSet)$, and $0\le r\le s$,
\begin{equation}
    0\le \phi_x(\mu,s)-\phi_x(\mu,r)\le L(s-r).
\end{equation}
\end{lemma}

\begin{proof}
Monotonicity follows from inclusion of Wasserstein balls. If $s=0$, then $r=0$ and the claim is immediate. Assume $s>0$ and set $\theta=r/s$. Fix $Q\in B_s(\mu)$ and define the mixture
\begin{equation}
    Q'=\theta Q+(1-\theta)\mu.
\end{equation}
Mixing an optimal (or arbitrarily near-optimal) coupling between $Q$ and $\mu$ with the identity coupling of $\mu$ shows that
\begin{equation}
    \dW(Q',\mu)\le\theta\dW(Q,\mu)\le r,
\end{equation}
so $Q'\in B_r(\mu)$. By the Kantorovich--Rubinstein inequality and Assumption \ref{ass:lipschitz},
\begin{align}
\E_Q[h(x,\xi)]-\E_{Q'}[h(x,\xi)]
&=(1-\theta)\bigl(\E_Q[h(x,\xi)]-\E_\mu[h(x,\xi)]\bigr)\notag\\
&\le (1-\theta)L\dW(Q,\mu)\le L(s-r).
\end{align}
Thus $\E_Q[h(x,\xi)]\le\phi_x(\mu,r)+L(s-r)$. Taking the supremum over $Q\in B_s(\mu)$ proves the result.
\end{proof}

\begin{theorem}[Joint center--radius stability]\label{thm:center-radius-stability}
Under Assumption \ref{ass:lipschitz}, for any $\mu,\nu\in\cM(\XiSet)$ and $r,s\ge0$,
\begin{equation}\label{eq:center-radius-stability}
    |V(\mu,r)-V(\nu,s)|
    \le L\bigl(\dW(\mu,\nu)+|r-s|\bigr).
\end{equation}
The same bound holds with $V$ replaced by $\phi_x$ for any fixed $x$.
\end{theorem}

\begin{proof}
Let $d=\dW(\mu,\nu)$. The triangle inequality implies
\begin{equation}
B_r(\mu)\subseteq B_{r+d}(\nu),
\end{equation}
so $\phi_x(\mu,r)\le\phi_x(\nu,r+d)$. If $r+d\ge s$, Lemma \ref{lem:radius-lipschitz} gives
\begin{equation}
\phi_x(\mu,r)-\phi_x(\nu,s)
\le L(r+d-s)\le L\bigl(d+|r-s|\bigr).
\end{equation}
If $r+d<s$, monotonicity instead gives $\phi_x(\mu,r)\le\phi_x(\nu,s)$, so the same upper bound holds. Interchanging $(\mu,r)$ and $(\nu,s)$ proves the fixed-decision bound. Since the bound is uniform in $x$, taking infima preserves it and proves \eqref{eq:center-radius-stability}.
\end{proof}

\begin{corollary}[Empirical versus predictive certificates]\label{cor:center-stability}
At a common radius $r$,
\begin{equation}
    |\widetilde J_N(r)-\widehat J_N(r)|
    \le L\dW(\Ptil,\Phat),
\end{equation}
where $\widehat J_N(r)$ denotes empirical-center DRO and $\widetilde J_N(r)$ denotes predictive-center DRO.
\end{corollary}

\subsection{Learning and scenario-discretization errors}

Suppose a predictive model outputs a probability law $\widehat G_N=G_{\widehat\theta_N}$ from a model class $\mathcal G=\{G_\theta:\theta\in\Theta\}$. Let $\theta^\star\in\argmin_{\theta\in\Theta}\dW(P,G_\theta)$ when it exists. The triangle inequality gives
\begin{equation}\label{eq:error-decomp}
    \dW(P,\widehat G_N)
    \le \underbrace{\inf_{\theta\in\Theta}\dW(P,G_\theta)}_{\text{approximation error}}
    +\underbrace{\dW(G_{\theta^\star},G_{\widehat\theta_N})}_{\text{estimation error}}.
\end{equation}
For optimization, a continuous $\widehat G_N$ is often replaced by a finite scenario law $\Ptil$. This creates a third error term.

\begin{proposition}[Additive scenario-discretization error]\label{prop:scenario-error}
For every predictive law $\widehat G_N$ and discrete approximation $\Ptil$,
\begin{equation}\label{eq:scenario-error}
    \dW(P,\Ptil)
    \le \dW(P,\widehat G_N)+\dW(\widehat G_N,\Ptil).
\end{equation}
If random bounds $r_N^{\mathrm{model}}$ and $r_{N,M}^{\mathrm{scen}}$ satisfy
\begin{align}
\Prob\{\dW(P,\widehat G_N)\le r_N^{\mathrm{model}}\}&\ge1-\beta_1,\\
\Prob\{\dW(\widehat G_N,\Ptil)\le r_{N,M}^{\mathrm{scen}}\}&\ge1-\beta_2,
\end{align}
then the radius $r_N^{\mathrm{model}}+r_{N,M}^{\mathrm{scen}}$ covers $P$ around $\Ptil$ with probability at least $1-\beta_1-\beta_2$.
\end{proposition}

\begin{proof}
Equation \eqref{eq:scenario-error} is the triangle inequality. On the intersection of the two stated events, the right-hand side is at most $r_N^{\mathrm{model}}+r_{N,M}^{\mathrm{scen}}$. The probability statement follows from the union bound.
\end{proof}

\begin{remark}
When $\Ptil$ is formed from $M$ i.i.d. draws from $\widehat G_N$, the scenario term is itself an empirical Wasserstein error conditional on the fitted model. It should either be bounded theoretically, estimated on an independent simulation sample, or made negligible through a convergence study in $M$. Treating generated scenarios as if they were the exact predictive law can lead to systematic undercoverage.
\end{remark}

A statistical generalization bound for $\widehat G_N$, combined with a bound for scenario discretization, can therefore be transferred directly to DRO through Theorems \ref{thm:pred-finite}, \ref{thm:pred-consistency}, and \ref{thm:pred-regret}.

\section{Learned Predictive Radii}\label{sec:learned-radius}

We now make the radius itself random and learned. Let
\begin{equation}
    \Pbar\in\{\Phat,\Ptil\}
\end{equation}
be a generic nominal law. Let
\begin{equation}\label{eq:random-radius}
    \epshat=g_\theta(Z_N)
\end{equation}
be a learned radius, where $g_\theta$ is constrained to take values in $\R_+$ and $Z_N$ contains dataset-level features. Define
\begin{equation}\label{eq:learned-radius-dro}
    \widehat J_N=
    \inf_{x\in X}\sup_{Q\in B_{\epshat}(\Pbar)}\E_Q[h(x,\xi)],
\end{equation}
with optimizer $\widehat x_N$.

Let
\begin{equation}\label{eq:Rn-generic}
    R_N=\dW(P,\Pbar)
\end{equation}
be the true Wasserstein error of the nominal center.

\subsection{Finite-sample validity of learned radii}

\begin{theorem}[Learned-radius finite-sample guarantee]\label{thm:learned-radius-finite}
If
\begin{equation}\label{eq:learned-radius-coverage}
    P^N\{R_N\le \epshat\}\ge1-\beta,
\end{equation}
then
\begin{equation}
    P^N\left\{\E_P[h(\widehat x_N,\xi)]\le \widehat J_N\right\}\ge1-\beta.
\end{equation}
\end{theorem}

\begin{proof}
On the event $\{R_N\le\epshat\}$, $P\in B_{\epshat}(\Pbar)$. Therefore
\begin{equation}
\E_P[h(\widehat x_N,\xi)]\le
\sup_{Q\in B_{\epshat}(\Pbar)}\E_Q[h(\widehat x_N,\xi)]
=\widehat J_N.
\end{equation}
The event has probability at least $1-\beta$.
\end{proof}

\begin{corollary}[Reliability is lower bounded by radius coverage]\label{cor:eta}
Define the undercoverage probability
\begin{equation}
    \eta_N=P^N\{R_N>\epshat\}.
\end{equation}
Then
\begin{equation}
    P^N\left\{\E_P[h(\widehat x_N,\xi)]\le\widehat J_N\right\}
    \ge 1-\eta_N.
\end{equation}
\end{corollary}

\subsection{Comparison with measure-concentration radius}

Assume for this subsection that the learned radius and the measure-concentration radius use the same nominal center $\Pbar$. This is immediate when $\Pbar=\Phat$. If $\Pbar=\Ptil$, then a separate concentration or calibration bound for the predictive center is required.

Let $J_N^{\mathrm{mc}}$ denote the DRO value obtained with deterministic radius $\epsmc(\beta)$:
\begin{equation}
    J_N^{\mathrm{mc}}=\inf_{x\in X}\sup_{Q\in B_{\epsmc(\beta)}(\Pbar)}\E_Q[h(x,\xi)].
\end{equation}

\begin{theorem}[Same validity and lower certificate]\label{thm:same-valid-lower-cert}
Assume
\begin{equation}
    P^N\{R_N\le\epshat\}\ge1-\beta
\end{equation}
and
\begin{equation}
    \epshat\le \epsmc(\beta)
\end{equation}
almost surely. Then
\begin{equation}
    \widehat J_N\le J_N^{\mathrm{mc}}
\end{equation}
almost surely, and
\begin{equation}
    P^N\left\{\E_P[h(\widehat x_N,\xi)]\le\widehat J_N\right\}\ge1-\beta.
\end{equation}
\end{theorem}

\begin{proof}
If $\epshat\le\epsmc(\beta)$, then
\begin{equation}
    B_{\epshat}(\Pbar)\subseteq B_{\epsmc(\beta)}(\Pbar).
\end{equation}
Thus for every $x$,
\begin{equation}
    \sup_{Q\in B_{\epshat}(\Pbar)}\E_Q[h(x,\xi)]
    \le
    \sup_{Q\in B_{\epsmc(\beta)}(\Pbar)}\E_Q[h(x,\xi)].
\end{equation}
Taking infimum over $x$ gives $\widehat J_N\le J_N^{\mathrm{mc}}$. The reliability follows from Theorem \ref{thm:learned-radius-finite}.
\end{proof}

\begin{theorem}[High-probability comparison]\label{thm:hp-comparison-mc}
Suppose
\begin{equation}
    P^N\{R_N\le\epshat\}\ge1-\beta,
    \qquad
    P^N\{\epshat\le\epsmc(\beta)\}\ge1-\alpha.
\end{equation}
Then
\begin{equation}
    P^N\left\{\E_P[h(\widehat x_N,\xi)]\le\widehat J_N\le J_N^{\mathrm{mc}}\right\}
    \ge1-\alpha-\beta.
\end{equation}
\end{theorem}

\begin{proof}
By the union bound, the events $\{R_N\le\epshat\}$ and $\{\epshat\le\epsmc(\beta)\}$ hold simultaneously with probability at least $1-\alpha-\beta$. On their intersection, Theorem \ref{thm:learned-radius-finite} gives $\E_P[h(\widehat x_N,\xi)]\le\widehat J_N$, and the set-inclusion argument in Theorem \ref{thm:same-valid-lower-cert} gives $\widehat J_N\le J_N^{\mathrm{mc}}$.
\end{proof}

\subsection{Regret and certificate-gap bounds with random radii}

\begin{theorem}[Random-radius regret bound]\label{thm:random-radius-regret}
Under Assumption \ref{ass:lipschitz},
\begin{equation}\label{eq:random-radius-general-regret}
    0\le\E_P[h(\widehat x_N,\xi)]-J^\star
    \le L(\epshat+R_N)
\end{equation}
on the event $R_N\le\epshat$. In particular,
\begin{equation}
    0\le\E_P[h(\widehat x_N,\xi)]-J^\star\le2L\epshat
\end{equation}
on the same event.
\end{theorem}

\begin{proof}
The proof is the same as Theorem \ref{thm:pred-regret}, with $\Ptil$, $r_N$, and $\varepsilon_N$ replaced by $\Pbar$, $R_N$, and $\epshat$.
\end{proof}

\begin{corollary}[Rate transfer]\label{cor:rate-transfer}
Suppose $P^N(R_N\le\epshat)\ge1-\beta_N$, where $\beta_N\to0$, and $\epshat=O_P(a_N)$ for a deterministic positive sequence $a_N$. Under Assumption \ref{ass:lipschitz},
\begin{equation}
    \E_P[h(\widehat x_N,\xi)]-J^\star=O_P(a_N).
\end{equation}
More explicitly, if $P^N(\epshat\le Ca_N)\ge1-\alpha_N$, then
\begin{equation}
P^N\left\{\E_P[h(\widehat x_N,\xi)]-J^\star\le2LCa_N\right\}
\ge1-\alpha_N-\beta_N.
\end{equation}
\end{corollary}

\begin{theorem}[Expected regret with undercoverage]\label{thm:expected-regret}
Assume Assumption \ref{ass:lipschitz} and suppose there is a finite constant $B$ such that
\begin{equation}
0\le \E_P[h(x,\xi)]-J^\star\le B\qquad\forall x\in X.
\end{equation}
Then, with $\eta_N=P^N(R_N>\epshat)$,
\begin{equation}\label{eq:expected-regret}
\E_{P^N}\left[\E_P[h(\widehat x_N,\xi)]-J^\star\right]
\le 2L\E_{P^N}[\epshat]+B\eta_N.
\end{equation}
\end{theorem}

\begin{proof}
On the coverage event, Theorem \ref{thm:random-radius-regret} bounds regret by $2L\epshat$. On its complement, regret is at most $B$. Splitting the expectation over the two events gives
\begin{equation}
\E[\mathrm{regret}]
\le2L\E[\epshat\1\{R_N\le\epshat\}]+B\eta_N
\le2L\E[\epshat]+B\eta_N.
\end{equation}
\end{proof}

\subsection{Asymptotic consistency with learned radii}

\begin{theorem}[Learned-radius asymptotic consistency]\label{thm:learned-radius-consistency}
Assume
\begin{equation}
    \epshat\to0\quad\text{almost surely},
\end{equation}
and
\begin{equation}
    R_N\le\epshat
\end{equation}
for all sufficiently large $N$, almost surely. Under Assumption \ref{ass:regularity},
\begin{equation}
    \widehat J_N\to J^\star
    \qquad\text{almost surely.}
\end{equation}
If the optimizer-regularity conditions in Assumption \ref{ass:regularity} hold, every accumulation point of $\{\widehat x_N\}$ is a true optimizer.
\end{theorem}

\begin{proof}
For any $Q_N\in B_{\epshat}(\Pbar)$,
\begin{equation}
    \dW(Q_N,P)\le\dW(Q_N,\Pbar)+\dW(\Pbar,P)
    \le\epshat+R_N\le2\epshat\to0
\end{equation}
eventually almost surely. The rest of the proof is identical to Theorem \ref{thm:pred-consistency}: evaluate the robust objective at a $\delta$-optimal true decision, take nearly worst-case $Q_N$, use upper semicontinuity and linear growth to pass to the limit, and use the sandwich $J^\star\le\E_P[h(\widehat x_N,\xi)]\le\widehat J_N$.
\end{proof}

\begin{corollary}[Borel--Cantelli version]\label{cor:bc}
Place the sequence of estimators on a common probability space, for example the infinite-product space generated by an i.i.d. sequence. Suppose there exists $\beta_N$ such that
\begin{equation}
    P^\infty\{R_N>\epshat\}\le\beta_N,
    \qquad
    \sum_{N=1}^\infty\beta_N<\infty,
\end{equation}
and $\epshat\to0$ almost surely. Then the conclusions of Theorem \ref{thm:learned-radius-consistency} hold.
\end{corollary}

\begin{proof}
By the Borel--Cantelli lemma, the event $\{R_N>\epshat\}$ occurs only finitely many times almost surely. Thus $R_N\le\epshat$ eventually almost surely. Apply Theorem \ref{thm:learned-radius-consistency}.
\end{proof}

\subsection{Oracle conditional-quantile radius}

Let $Z_N$ be observable features used by the radius model. Define
\begin{equation}
    R_N=\dW(P,\Pbar).
\end{equation}
The oracle conditional radius is
\begin{equation}\label{eq:oracle-radius}
    \varepsilon_N^\star(Z_N;\beta)=
    \inf\{r\ge0:P^N(R_N\le r\mid Z_N)\ge1-\beta\}.
\end{equation}

\begin{theorem}[Validity and pointwise optimality of the oracle conditional radius]\label{thm:oracle-radius}
Assume a measurable version of the conditional quantile in \eqref{eq:oracle-radius} is used. Then
\begin{equation}
    P^N\{R_N\le \varepsilon_N^\star(Z_N;\beta)\mid Z_N\}\ge1-\beta
    \qquad\text{almost surely},
\end{equation}
and therefore
\begin{equation}
    P^N\{R_N\le \varepsilon_N^\star(Z_N;\beta)\}\ge1-\beta.
\end{equation}
Moreover, if $r(Z_N)$ is any nonnegative, $Z_N$-measurable rule satisfying
\begin{equation}
    P^N\{R_N\le r(Z_N)\mid Z_N\}\ge1-\beta
    \qquad\text{almost surely},
\end{equation}
then
\begin{equation}
    r(Z_N)\ge \varepsilon_N^\star(Z_N;\beta)
    \qquad\text{almost surely}.
\end{equation}
Thus the oracle conditional quantile minimizes every increasing functional of the radius, including its expectation, among conditionally valid radius rules.
\end{theorem}

\begin{proof}
The coverage statement follows from the generalized-inverse definition of a conditional quantile and right-continuity of the conditional distribution function. Taking expectations over $Z_N$ gives unconditional coverage. For minimality, fix a feature realization outside a null set. Any value smaller than the infimum in \eqref{eq:oracle-radius} has conditional coverage strictly below $1-\beta$; hence every conditionally valid rule must be at least the conditional quantile. The final statement follows by monotonicity.
\end{proof}

\subsection{Approximate learned radii and safety margins}

\begin{theorem}[Approximate oracle domination]\label{thm:approx-oracle}
Let $\varepsilon_N^\star=\varepsilon_N^\star(Z_N;\beta)$ be the oracle conditional radius. If
\begin{equation}
    P^N\{\epshat\ge \varepsilon_N^\star\}\ge1-\alpha,
\end{equation}
then
\begin{equation}
    P^N\{R_N\le\epshat\}\ge1-\beta-\alpha.
\end{equation}
Consequently,
\begin{equation}
    P^N\{\E_P[h(\widehat x_N,\xi)]\le\widehat J_N\}\ge1-\beta-\alpha.
\end{equation}
\end{theorem}

\begin{proof}
Let $A_N=\{\epshat\ge\varepsilon_N^\star\}$. Then
\begin{equation}
    A_N\cap\{R_N\le\varepsilon_N^\star\}\subseteq\{R_N\le\epshat\}.
\end{equation}
Therefore
\begin{align*}
P^N(R_N\le\epshat)
&\ge P^N(A_N\cap\{R_N\le\varepsilon_N^\star\})\\
&\ge P^N(R_N\le\varepsilon_N^\star)-P^N(A_N^c)\\
&\ge1-\beta-\alpha.
\end{align*}
Apply Theorem \ref{thm:learned-radius-finite}.
\end{proof}

\begin{corollary}[Safety-margin correction]\label{cor:safety-margin}
If
\begin{equation}
    P^N\{|\epshat-\varepsilon_N^\star|\le\Delta_N\}\ge1-\alpha,
\end{equation}
and we define
\begin{equation}
    \epshat^+=\epshat+\Delta_N,
\end{equation}
then
\begin{equation}
    P^N\{R_N\le\epshat^+\}\ge1-\beta-\alpha.
\end{equation}
\end{corollary}

\subsection{Split-conformal calibration of learned radii}

The previous results assume access to or approximation of the oracle radius. A practical finite-sample calibration method is split conformal prediction across exchangeable problem instances, such as independently simulated environments or independently sampled tasks. The base predictor must be fitted on a proper training set that is separate from the calibration instances.

Suppose the fitted predictor $f_\theta$ is fixed, and let $t=1,\ldots,T$ index calibration instances. For each instance, observe features $Z_t$ and a target
\begin{equation}
    R_t=\dW(P_t,\bar P_t),
\end{equation}
and compute the one-sided residual
\begin{equation}
    S_t=R_t-f_\theta(Z_t).
\end{equation}
Let $S_{(1)}\le\cdots\le S_{(T)}$ be the ordered residuals and define
\begin{equation}\label{eq:conformal-index}
    k=\left\lceil (T+1)(1-\beta)\right\rceil,
    \qquad
    q_{T,\beta}=\begin{cases}
    S_{(k)},&k\le T,\\
    +\infty,&k=T+1.
    \end{cases}
\end{equation}
The calibrated radius is
\begin{equation}\label{eq:conformal-radius}
    \widehat\varepsilon_{\mathrm{new}}
    =\max\{0,f_\theta(Z_{\mathrm{new}})+q_{T,\beta}\}.
\end{equation}

\begin{theorem}[Finite-sample split-conformal radius]\label{thm:conformal}
Assume that, conditional on the proper training data used to fit $f_\theta$, the $T$ calibration pairs and the new pair are exchangeable. Then
\begin{equation}
    \Prob\{R_{\mathrm{new}}\le \widehat\varepsilon_{\mathrm{new}}\}\ge1-\beta.
\end{equation}
Consequently, the corresponding DRO certificate satisfies
\begin{equation}
    \Prob\left\{\E_{P_{\mathrm{new}}}[h(\widehat x_{\mathrm{new}},\xi)]\le \widehat J_{\mathrm{new}}\right\}\ge1-\beta.
\end{equation}
\end{theorem}

\begin{proof}
Conditional on the proper training data, the residuals $S_1,\ldots,S_T,S_{\mathrm{new}}$ are exchangeable. The rank of $S_{\mathrm{new}}$ among the $T+1$ residuals is therefore uniform up to ties. With the order statistic in \eqref{eq:conformal-index},
\begin{equation}
    \Prob\{S_{\mathrm{new}}\le q_{T,\beta}\}\ge \frac{k}{T+1}\ge1-\beta.
\end{equation}
This event implies $R_{\mathrm{new}}\le f_\theta(Z_{\mathrm{new}})+q_{T,\beta}$, and truncation at zero cannot reduce coverage because $R_{\mathrm{new}}\ge0$. The certificate follows from Theorem \ref{thm:learned-radius-finite}.
\end{proof}

\begin{corollary}[Conformal calibration with conservative proxy labels]\label{cor:conformal-proxy}
Suppose the true target $R_t$ is not observed exactly and the conformal procedure is instead run on proxy targets $U_t$. Conditional on the proper training data, assume that the $T$ calibration pairs $(Z_t,U_t)$ and the new pair $(Z_{\mathrm{new}},U_{\mathrm{new}})$ are exchangeable. Construct $\widehat\varepsilon_{\mathrm{new}}$ from residuals $U_t-f_\theta(Z_t)$ using \eqref{eq:conformal-index}. If
\begin{equation}
    \Prob\{R_{\mathrm{new}}\le U_{\mathrm{new}}\}\ge1-\gamma,
\end{equation}
then
\begin{equation}
    \Prob\{R_{\mathrm{new}}\le\widehat\varepsilon_{\mathrm{new}}\}
    \ge1-\beta-\gamma.
\end{equation}
In particular, if $U_t\ge R_t$ almost surely for every task, the original $1-\beta$ guarantee is preserved.
\end{corollary}

\begin{proof}
The split-conformal rank argument applied to the proxy targets gives
\begin{equation}
    \Prob\{U_{\mathrm{new}}\le\widehat\varepsilon_{\mathrm{new}}\}\ge1-\beta.
\end{equation}
On the intersection of this event with $\{R_{\mathrm{new}}\le U_{\mathrm{new}}\}$, the desired coverage event holds. The union bound proves the result.
\end{proof}

\begin{remark}[Time-series data]
Ordinary split conformal is not justified by a chronological rolling sequence merely because the observations are placed in different folds. In the experiments, ordinary split conformal is restricted to independent task draws. The sequential study uses weighted and adaptive procedures designed for drift \citep{barber2023conformal,gibbs2021adaptive}, while the rolling financial study uses chronological pseudo-tasks and reports operational diagnostics rather than exchangeable finite-sample coverage.
\end{remark}

\subsection{Distributional radius versus certificate radius}

Distributional coverage $P\in B_\varepsilon(\Pbar)$ is sufficient but not necessary for the DRO certificate. Define
\begin{align}
C_N(\varepsilon)&=\{P\in B_\varepsilon(\Pbar)\},\\
G_N(\varepsilon)&=\{\E_P[h(x_N(\varepsilon),\xi)]\le J_N(\varepsilon)\}.
\end{align}
Here $x_N(\varepsilon)$ and $J_N(\varepsilon)$ are the optimizer and value of the DRO problem with radius $\varepsilon$; a measurable optimizer selection is assumed whenever it is needed for probability statements.

Since $C_N(\varepsilon)\subseteq G_N(\varepsilon)$, define
\begin{align}
\rdist(\beta)&=\inf\{\varepsilon:P^N(C_N(\varepsilon))\ge1-\beta\},\\
\rcert(\beta)&=\inf\{\varepsilon:P^N(G_N(\varepsilon))\ge1-\beta\}.
\end{align}

\begin{proposition}[Certificate radii can be smaller]\label{prop:cert-radius}
For every $\beta\in(0,1)$,
\begin{equation}
    \rcert(\beta)\le\rdist(\beta).
\end{equation}
\end{proposition}

\begin{proof}
Because $C_N(\varepsilon)\subseteq G_N(\varepsilon)$,
\begin{equation}
    P^N(G_N(\varepsilon))\ge P^N(C_N(\varepsilon))
\end{equation}
for every $\varepsilon$. Hence every distributionally valid radius is certificate-valid. Taking infima proves the claim.
\end{proof}

\begin{remark}
A model trained to predict $\rdist(\beta)$ is universal because it targets the true distribution. A model trained to predict $\rcert(\beta)$ can be less conservative but is task-specific: it depends on $h$, $X$, and the optimization problem.
\end{remark}

\section{Predictive DRO Reformulations for Portfolio Optimization}\label{sec:portfolio-reform}

This section records the optimization problem used in the experiments below. Consider $m$ assets with return vector $\xi\in\R^m$. Let
\begin{equation}
    X=\{x\in\R_+^m:\mathbf 1^\top x=1\}
\end{equation}
be the long-only simplex. The mean-CVaR loss is
\begin{equation}\label{eq:mean-cvar}
\inf_{x\in X}\left\{\E_P[-x^\top\xi]+\rho\,\mathrm{CVaR}_{\alpha,P}(-x^\top\xi)\right\}.
\end{equation}
Here $\alpha\in(0,1)$ denotes the upper-tail probability (thus confidence level $1-\alpha$). Using the Rockafellar--Uryasev representation \citep{rockafellar2000optimization},
\begin{equation}
\mathrm{CVaR}_{\alpha,P}(Y)=\inf_{\tau\in\R}\E_P\left[\tau+\frac1\alpha(Y-\tau)_+\right],
\end{equation}
we rewrite \eqref{eq:mean-cvar} as
\begin{equation}\label{eq:portfolio-loss}
\inf_{x\in X,\tau\in\R}\E_P\left[\max_{k=1,2}\left\{a_k x^\top\xi+b_k\tau\right\}\right],
\end{equation}
where
\begin{equation}
    a_1=-1,
    \qquad
    a_2=-1-\frac{\rho}{\alpha},
    \qquad
    b_1=\rho,
    \qquad
    b_2=\rho\left(1-\frac1\alpha\right).
\end{equation}
For a nominal distribution $\Pbar=\sum_{j=1}^Mp_j\delta_{\zeta_j}$ and Wasserstein radius $\varepsilon$, the DRO counterpart over $\XiSet=\R^m$ is the convex program
\begin{equation}\label{eq:portfolio-dro-lp}
\begin{aligned}
\inf_{x,\tau,\lambda,s_j}\quad
&\lambda\varepsilon+\sum_{j=1}^Mp_js_j\\
\mathrm{s.t.}\quad
&x\in X,\\
&a_k x^\top\zeta_j+b_k\tau\le s_j,
\qquad j=1,\ldots,M,\quad k=1,2,\\
&\norm{a_kx}_*\le\lambda,
\qquad k=1,2.
\end{aligned}
\end{equation}
If the Wasserstein ground norm is $\ell_1$, then the dual norm is $\ell_\infty$, and \eqref{eq:portfolio-dro-lp} is a linear program after writing $\norm{a_kx}_\infty\le\lambda$ as linear inequalities. A deterministic proportional transaction-cost term $\kappa\norm{x-x^-}_1$, where $x^-$ is the pre-trade portfolio, can be added directly to the objective without changing the inner Wasserstein reformulation. This extension is used in the rolling experiment.

\section{Computational Experiments}\label{sec:experiments}

The experiments ask whether learning the center and radius improves three distinct outcomes: approximation of the data-generating law, calibration of the resulting ambiguity set, and the downstream decision.  Keeping these outcomes separate is important because a better calibrated ambiguity set need not change the optimizer, and a better predictive score need not reduce decision loss.  The principal empirical conclusion is correspondingly qualified: learned radii can materially improve reliability, but calibration alone does not deliver a uniform efficiency or decision-value advantage.

\subsection{Protocol, estimands, and claim taxonomy}

For an independent task $t$, the primary quantities are
\begin{equation}\label{eq:experimental-estimands}
\begin{aligned}
C_{\mathrm{dist}}
&=\Prob\!\left\{\dW(P_t,\bar P_t)\le\widehat\varepsilon_t\right\},
&
C_{\mathrm{cert}}
&=\Prob\!\left\{\E_{P_t}[h(\widehat x_t,\xi)]\le\widehat J_t\right\},
\\
\mathcal R_t
&=\E_{P_t}[h(\widehat x_t,\xi)]-J_t^\star.
\end{aligned}
\end{equation}
The terminology in the results reflects what is actually identified.  We use \emph{one-dimensional numerical containment} when $W_1$ is evaluated on the frozen deterministic quantile grid with the saved sampling and boundary allowance.  Because grid error is not formally certified, these outcomes remain numerical containment diagnostics rather than exact population-coverage observations.  Multivariate synthetic containment is evaluated with a 64-point debiased-Sinkhorn discrepancy and is therefore called \emph{proxy containment}, not exact Wasserstein coverage.  In the financial experiment $P_t$ is unknown; future blocks are noisy proxy targets and $C_{\mathrm{cert}}$ is replaced by an operational certificate diagnostic.  No real-data distributional-coverage claim is made.

Training, calibration, validation or scenario-study, and test information were separated according to the frozen protocol of each experiment.  In the newsvendor and synthetic studies, centers were fitted only from each task's training sample, while radius-training, calibration, performance-validation or scenario-study, and test tasks remained disjoint.  The contextual experiment additionally used a separate center-training role.  Test tasks were evaluated once after center selection, radius models, safety corrections, and the three primary comparisons were fixed.  Oracle radii use test information and appear only as unattainable benchmarks.  Common random numbers pair methods within a task.

\begin{table}[H]
\centering
\caption{Scope and statistical meaning of the four experimental families.  ``Core'' denotes the prespecified primary configuration; the remaining axes are one-factor sensitivity designs.}
\label{tab:experiment-scope}
\small
\setlength{\tabcolsep}{4pt}
\begin{tabularx}{\linewidth}{@{}>{\raggedright\arraybackslash}p{.16\linewidth}>{\raggedright\arraybackslash}p{.19\linewidth}>{\raggedright\arraybackslash}X>{\raggedright\arraybackslash}p{.24\linewidth}@{}}
\toprule
Experiment & Evaluation units & Primary reliability object & Main variation \\
\midrule
Newsvendor & 25,000 held-out tasks & one-dimensional numerical $W_1$ & 5 laws, 5 values of $N$ \\
Synthetic portfolio & 4,000 core tasks/method & 64-point Sinkhorn proxy & $m,N/m,\alpha,\rho$, norm, contamination \\
Contextual shift & 4,000 tasks; 1,500 time points & conditional discrepancy proxy & interpolation, shift, misspecification \\
Rolling portfolios & 216 rebalances/cell & future-block and operational proxies & window, cost, risk, norm, universe \\
\bottomrule
\end{tabularx}
\end{table}

The primary target is $1-\beta=0.95$; 0.90 and 0.99 are sensitivity levels.  Binary task-level intervals are 95\% Wilson intervals.  Mean paired differences, including binary containment differences, use 2,000 paired bootstrap draws; two-sided Pratt--Wilcoxon tests assess paired shifts.  The bootstrap interval and Wilcoxon test concern different functionals and are reported separately.  Holm adjustment is applied to prespecified secondary comparisons, while the three primary synthetic-portfolio comparisons are reported without a multiplicity adjustment.  Financial return intervals use a 21-day moving-block bootstrap with 2,000 draws.  Sequential observations come from one dependent trajectory and are not assigned ordinary binomial confidence intervals.

The master seed is 20260711.  Linear programs use SciPy HiGHS and nonlinear comparator problems use SLSQP, both with a reported feasibility tolerance of $10^{-8}$.  The complete run used 12 workers, Python 3.11.7, and an NVIDIA RTX 4060 Laptop GPU, and required 331,302 seconds (92.0 hours) across the four stages.  The replication archive contains task-level data, fitted-model metadata, frozen configuration files, and cryptographic hashes; 90 automated tests and all 75 publication-artifact hashes pass.

\subsection{Center learning and finite-scenario approximation}

Figure~\ref{fig:center-scenario} separates predictive-model error from scenario error.  Averaged over the five newsvendor families, the GMM center's numerical model error falls from 6.941 at $N=20$ to 1.530 at $N=500$, a 78.0\% reduction.  The Gaussian and Student-$t$ shrinkage errors fall more modestly, from 9.193 to 7.114 and from 8.480 to 5.755, respectively.  Center learning is therefore useful in some families, but the advantage is not uniform across model classes.

\begin{figure}[!tbp]
\centering
\includegraphics[width=\linewidth,height=.78\textheight,keepaspectratio]{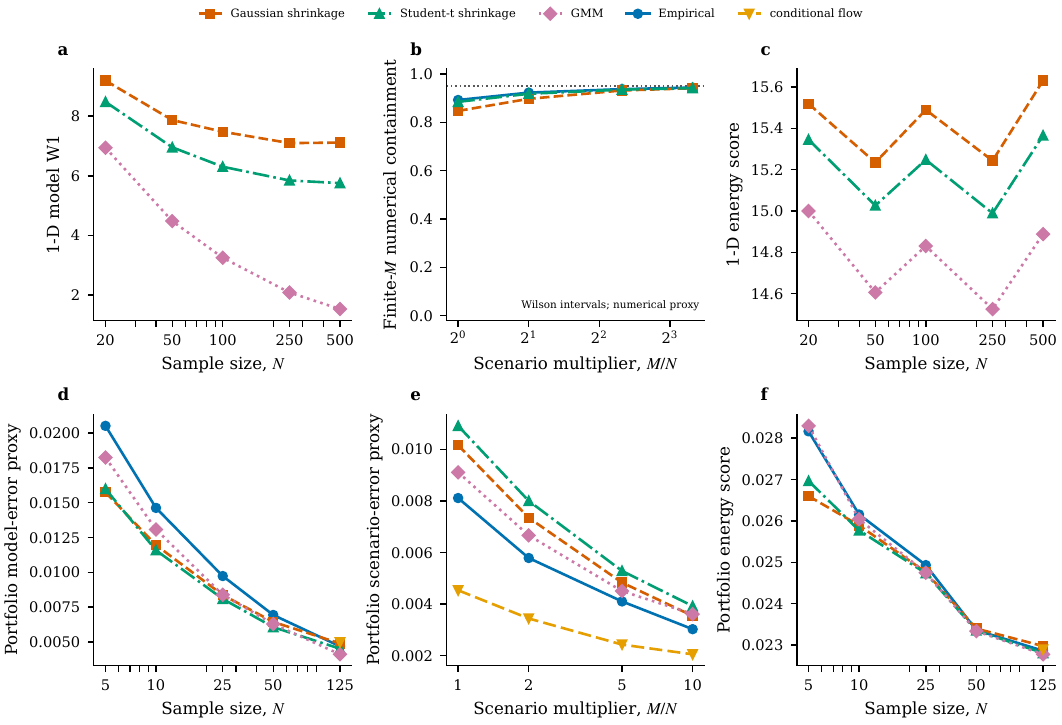}
\caption{Predictive-law and finite-scenario diagnostics.  Panels (a)--(c) report, for the one-dimensional newsvendor study, numerical model-$W_1$ error against sample size, fixed-radius finite-$M$ numerical containment against $M/N$, and the energy score.  Panels (d)--(f) report the corresponding multivariate portfolio model-error proxy, scenario-error proxy, and energy score.  The multivariate quantities are numerical discrepancy proxies, not exact Wasserstein distances.}
\label{fig:center-scenario}
\end{figure}

Scenario discretization remains material after the center is fitted.  At the 95\% target, fixed-radius numerical containment increases from 87.491\% at $M=N$ to 91.287\%, 93.495\%, and 94.225\% at $M/N=2,5,10$, respectively.  The direction agrees with the decomposition in Section~\ref{sec:predictive-dist}, but $M=10N$ does not eliminate undercoverage.  Thus a generated scenario law cannot be treated as the exact predictive law merely because its model error is small.

\subsection{Controlled calibration and decision value}

\subsubsection{One-dimensional newsvendor}

We use $c_h=1$, $c_b=4$, $N\in\{20,50,100,250,500\}$, and Gaussian, lognormal, Student-$t$, mixture, and contaminated demand.  Each distribution--sample-size cell contains 1,000 untouched test tasks.  The pooled estimand is marginal over this balanced, prespecified mixture; it is not a cellwise coverage statement.

Panel A of Table~\ref{tab:controlled-results} shows that split conformal attains 95.506\% pooled numerical containment.  The raw GBR radius already attains 95.575\%, however, and both rules produce the same mean regret, 1.2223.  Conformal is substantially less conservative than the fixed radius but is larger and has higher regret than empirical bootstrap.  In the secondary paired GMM-versus-empirical-bootstrap comparison, conformal increases containment by 2.956 percentage points (95\% paired interval 2.496--3.390; $p=6.73\times10^{-38}$) and increases regret by 0.6013 (0.5027--0.7012; $p=2.63\times10^{-240}$).  This is a calibration improvement, not a decision improvement.

\begin{table}[H]
\centering
\caption{Controlled 95\%-target results.  Intervals are 95\% Wilson intervals.  Panel A uses one-dimensional numerical $W_1$ containment and has 25,000 test tasks; its four determinate denominators are 24,983, 24,994, 24,991, and 25,000 in table order.  Panel B uses the saved 64-point debiased-Sinkhorn proxy.  Neither calculation is claimed to be exact population distributional coverage.}
\label{tab:controlled-results}
\small
\setlength{\tabcolsep}{5pt}
\begin{tabular}{@{}lrrr@{}}
\toprule
Method & Containment [95\% interval] & Mean radius & Mean regret \\
\midrule
\multicolumn{4}{@{}l}{\textit{Panel A: newsvendor, 25,000 test tasks}}\\
Empirical + bootstrap & 91.054 [90.694, 91.402]\% & 6.4756 & 0.6783 \\
Predictive + raw GBR & 95.575 [95.313, 95.823]\% & 9.9754 & 1.2223 \\
Predictive + split conformal & 95.506 [95.242, 95.756]\% & 9.9558 & 1.2223 \\
Predictive + fixed radius & 99.068 [98.941, 99.180]\% & 26.5597 & 1.2223 \\
\addlinespace
\multicolumn{4}{@{}l}{\textit{Panel B: synthetic mean--CVaR portfolio, 4,000 core tasks per method}}\\
Empirical + bootstrap & 0.000 [0.000, 0.096]\% & 0.02936 & 0.00329244 \\
Predictive + raw radius & 80.850 [79.601, 82.040]\% & 0.12956 & 0.00329341 \\
Predictive + split conformal & 95.275 [94.573, 95.890]\% & 0.13761 & 0.00329341 \\
Predictive + fixed radius & 100.000 [99.904, 100.000]\% & 1.11947 & 0.00329341 \\
\bottomrule
\end{tabular}
\end{table}

The invariance across predictive radii is structural rather than a solver artifact: with unrestricted support, the one-dimensional $W_1$ newsvendor reformulation adds the radius penalty to the robust value without changing its minimizer.  Consistent with this structure, all seven candidate radii tied and produced the same decision in every one of the 100 performance-validation cells.  Conditional diagnostics are less favorable than the pooled average: cellwise conformal rates range from 84.384\% to 98.298\%, and several contaminated small-$N$ cells have Wilson intervals entirely below 95\%.  Figure~\ref{fig:radius-calibration} therefore illustrates marginal calibration together with its dimensional and volatility-conditional limitations.

\begin{figure}[!tbp]
\centering
\includegraphics[width=\linewidth,height=.78\textheight,keepaspectratio]{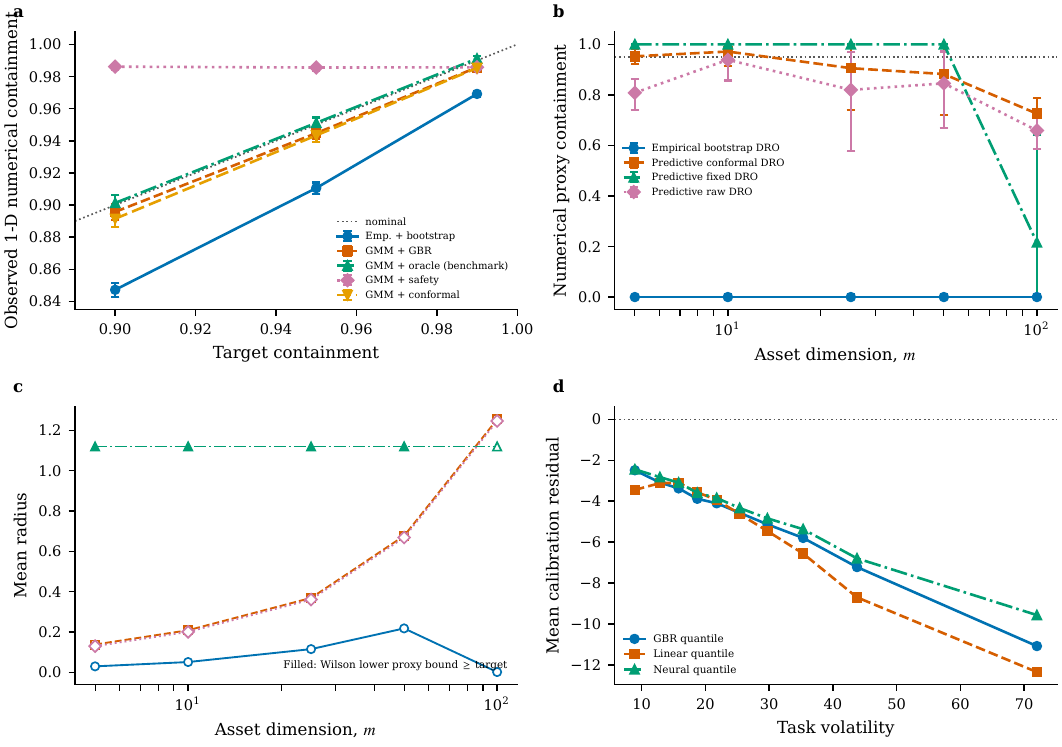}
\caption{Radius-calibration diagnostics.  Panel (a) reports pooled one-dimensional numerical containment at the three target levels.  Panel (b) reports multivariate proxy containment by dimension for selected portfolio methods.  Panel (c) reports mean radius, with filled markers indicating that the Wilson lower proxy-containment bound reaches the target.  Panel (d) relates newsvendor calibration residuals to task volatility.  Panels (b)--(c) concern a numerical Sinkhorn proxy rather than formal multivariate coverage.}
\label{fig:radius-calibration}
\end{figure}

For the saved numerical laws, all 1,714,585 recorded theorem inequalities hold at the stated tolerances: 1,019,585 normalized-regret checks, 347,500 scenario-decomposition checks, and 347,500 center--radius stability checks.  These diagnostics validate implementation of the bounds on the stored numerical objects; they do not certify exact population transport, remove approximation error, or alter the negative decision-value comparison above.

\subsubsection{Synthetic mean--CVaR portfolios}

The core portfolio design uses $m=5$, $N/m=5$, CVaR tail probability $\alpha=0.10$, risk weight $\rho=1$, the $\ell_1$ ground norm, and 1,000 paired tasks from each of four return families.  One-factor sweeps cover $m\in\{5,10,25,50,100\}$, $N/m\in\{1,2,5,10,25\}$, three tail probabilities, three risk weights, three ground norms, and four contamination rates.

Panel B of Table~\ref{tab:controlled-results} records a large proxy-calibration effect.  Conformal increases proxy containment by 14.425 percentage points relative to the raw radius and uses a radius 87.7\% smaller than the fixed rule.  Relative to empirical bootstrap, however, it uses a much larger radius.  All four methods have 100\% Monte Carlo certificate reliability, and the conformal-minus-bootstrap paired-regret difference is only $9.68\times10^{-7}$, with a paired interval approximately $[0,2.46\times10^{-6}]$, effect size $d_z=0.022$, and Wilcoxon $p=0.240$.  The experiment detects a difference in the saved discrepancy proxy but no statistically detectable improvement in paired Monte Carlo regret under the prespecified test.

This conclusion requires a material numerical qualification.  The proxy uses entropic regularization $0.01$, at most 5,000 Sinkhorn iterations, and a strict $10^{-7}$ marginal-residual criterion \citep{peyre2019computational,feydy2019interpolating}.  Across 27,200 radius-label tasks, 66.8\% of predictive and 65.8\% of empirical solves meet this criterion, and only 52.0\% meet it for both.  In a separate ten-case extended-iteration check, the saved proxy changes by at most $8.11\times10^{-6}$.  In the 12-case exact-assignment sensitivity, however, only 10 cases converge at the main regularization; among those 10, mean absolute error is 0.012177 and mean relative error is 25.368\%.  The converged-only core conformal rate is 95.729\%, compared with 95.275\% in the full sample, but this post hoc subset does not repair nonconverged training and calibration labels.  We therefore interpret every multivariate containment result only as an entropically regularized proxy.  Figure~\ref{fig:sinkhorn-appendix} reports the complete sensitivity analysis.  All 337,600 optimization rows are feasible within $9.998\times10^{-9}$.

\subsection{Contextual and sequential distribution shift}

The independent contextual experiment draws 1,000 tasks in each of four conditions: interpolation, an unseen-low covariate region, an unseen-high region, and deliberate predictive misspecification.  In this benchmark, the fitted conditional RealNVP does not yield an aggregate advantage over the unconditional empirical center.  Conditional RealNVP with split conformal attains 70.325\% proxy containment and mean regret 0.002374, compared with 88.800\% and 0.002162 for unconditional empirical split conformal.  The paired conditional-minus-unconditional effects are $-18.475$ percentage points in containment (95\% interval $[-19.925,-17.075]$; $p=8.51\times10^{-119}$) and $2.121\times10^{-4}$ in regret ($[1.915,2.326]\times10^{-4}$; $p=1.13\times10^{-80}$).

\begin{table}[H]
\centering
\caption{Independent contextual-shift results at the 95\% target.  All reliability entries are numerical discrepancy-proxy rates over 1,000 test tasks per condition; regret is averaged over all 4,000 tasks.}
\label{tab:contextual-results}
\footnotesize
\setlength{\tabcolsep}{2.7pt}
\begin{tabular}{@{}lrrrrrr@{}}
\toprule
Method & Aggregate & Interpolation & Misspecified & Unseen high & Unseen low & Mean regret \\
\midrule
Unconditional empirical + conformal & 88.800\% & 95.3\% & 68.9\% & 95.5\% & 95.5\% & 0.002162 \\
Conditional RealNVP + conformal & 70.325\% & 94.7\% & 19.4\% & 76.2\% & 91.0\% & 0.002374 \\
Conditional RealNVP + safety & 81.375\% & 99.1\% & 33.2\% & 95.6\% & 97.6\% & 0.002384 \\
\bottomrule
\end{tabular}
\end{table}

The failure is concentrated in extrapolation and misspecification.  Conditional conformal is close to target under interpolation (94.7\%) but falls to 91.0\% in the unseen-low region, 76.2\% in the unseen-high region, and 19.4\% under misspecification.  The safety margin raises aggregate containment to only 81.375\%.  Figure~\ref{fig:context-independent} shows that a smaller conditional radius is not evidence of greater reliability.

\begin{figure}[!tbp]
\centering
\includegraphics[width=\linewidth,height=.78\textheight,keepaspectratio]{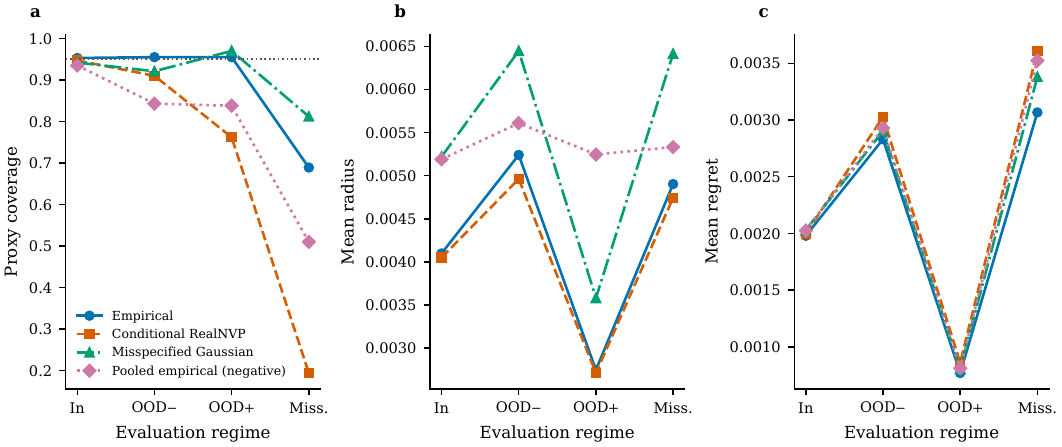}
\caption{Independent contextual-shift experiment.  The panels report numerical proxy containment, radius, regret, and reliability diagnostics across interpolation, unseen covariate regions, and predictive misspecification.  Error bars are 95\% Wilson intervals for binary task-level quantities.  These are conditional multivariate discrepancy proxies, not exact distributional coverage.}
\label{fig:context-independent}
\end{figure}

The 1,500-step sequential trajectory produces a different conclusion.  The adaptive controller covers 1,423 time points (94.867\%) and recovers 123 observations after the shift; exponential weighting covers 1,420 (94.667\%) and recovers after 292.  The fixed exchangeable-split diagnostic covers only 890 points (59.333\%) and has no recorded recovery.  Figure~\ref{fig:context-sequential} shows that both adaptive procedures expand the radius after the abrupt change and contract it as the system stabilizes.  Because these observations belong to one dependent trajectory, the reported rates are adaptation diagnostics under the corresponding sequential assumptions, not exchangeable split-conformal guarantees.

\begin{figure}[!tbp]
\centering
\includegraphics[width=\linewidth,height=.78\textheight,keepaspectratio]{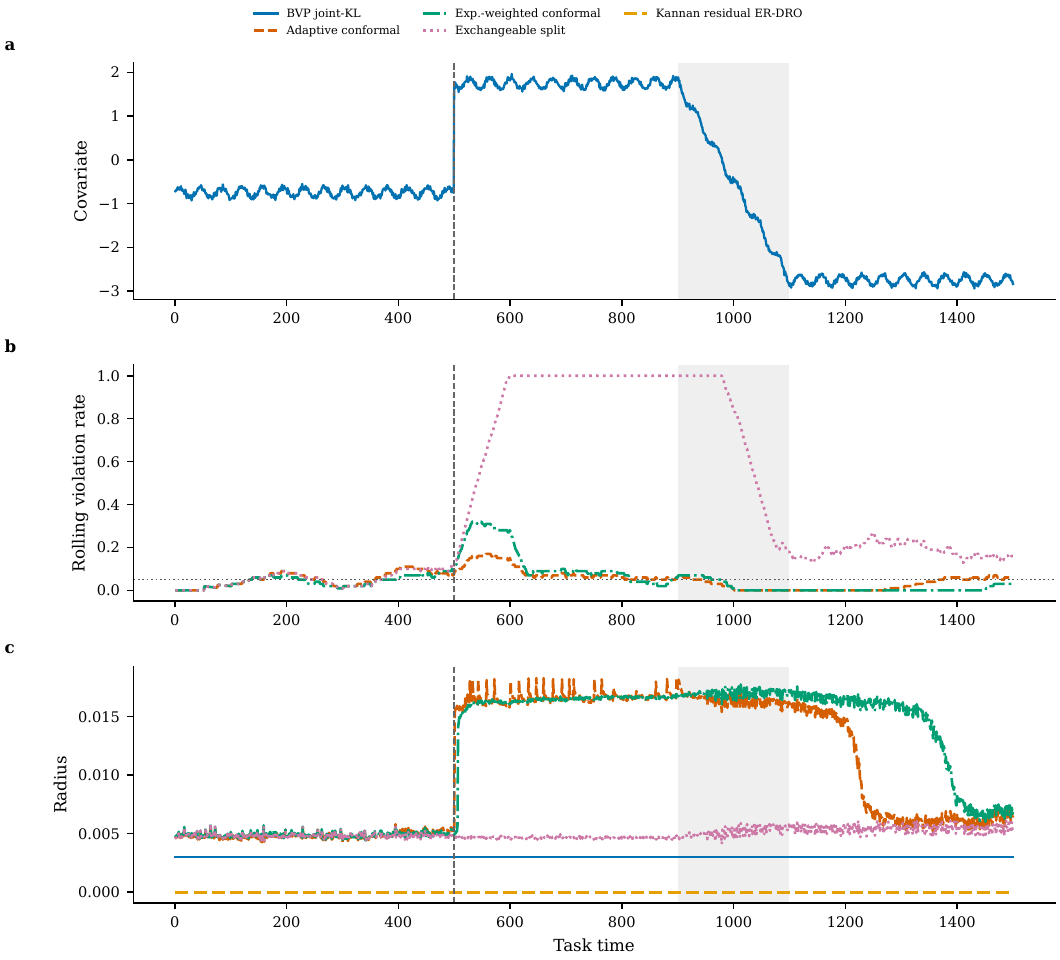}
\caption{Sequential adaptation under abrupt and gradual regime changes.  The panels align the latent regime, time-indexed proxy-containment indicators, rolling violation rates, and adaptive radii.  Recovery time is measured from the frozen abrupt-shift date.  Weighted and adaptive procedures are evaluated under nonexchangeable sequential assumptions; the fixed split rule is a diagnostic only.}
\label{fig:context-sequential}
\end{figure}

\FloatBarrier
\subsection{Rolling real-data portfolios}

The real-data study uses daily value-weighted Fama--French 49-industry returns from the \href{https://mba.tuck.dartmouth.edu/pages/faculty/ken.french/data_library.html}{Kenneth R. French Data Library} and adjusted-close prices from the Yahoo Finance public chart endpoint for SPY, EFA, EEM, IEF, TLT, LQD, SHY, TIP, VNQ, and GLD.  The available raw responses, all processed panels, and acquisition metadata are archived; the historical TIP raw-payload mismatch is documented below.  Data are intersected without return imputation.  The primary backtest spans December 2007 through December 2025, with a 756-day trailing training window, 126-day pseudo-task training blocks, 21-day proxy-label blocks, six chronological performance-validation tasks, 12 chronological calibration tasks, monthly rebalancing, and next-calendar-month evaluation.  The main specification uses $\alpha=0.10$, $\rho=1$, the $\ell_1$ ground norm, and 10 basis-point transaction costs; 5 and 25 basis points are sensitivities.

Table~\ref{tab:real-results} reports the main decision and reliability outcomes.  Gaussian safety reaches future-block proxy rates of 93.981\% for FF49 and 95.370\% for ETFs.  Its operational reliability is 91.667\% for FF49, whose Wilson interval $[87.213,94.664]\%$ lies below the target, and 98.611\% for ETFs, with interval $[95.997,99.527]\%$.  The corresponding net returns are 10.380\% (21-day moving-block interval 0.959--20.773\%) and 5.463\% (0.858--9.987\%).  These return intervals overlap competing methods; no stored paired test supports return, Sharpe, or risk dominance.

\begin{table}[H]
\centering
\caption{Rolling portfolio results at 10 basis-point transaction costs.  Proxy and operational columns are task rates over 216 monthly rebalances.  Return, volatility, Sharpe ratio, and maximum drawdown use net daily returns.  Dashes suppress radius and operational-certificate diagnostics for non-DRO benchmarks, even where an auxiliary nominal-objective diagnostic is available in the replication tables.}
\label{tab:real-results}
\footnotesize
\setlength{\tabcolsep}{2.6pt}
\begin{tabular}{@{}llrrrrrr@{}}
\toprule
Universe & Method & Proxy & Operational & Return & Volatility & Sharpe & Max. drawdown \\
\midrule
FF49 & Equal weight & -- & -- & 10.25\% & 20.96\% & 0.507 & $-55.30$\% \\
FF49 & Empirical bootstrap & 0.00\% & 80.56\% & 8.99\% & 16.85\% & 0.517 & $-47.81$\% \\
FF49 & Predictive fixed & 79.17\% & 88.43\% & 10.40\% & 20.18\% & 0.526 & $-53.63$\% \\
FF49 & Gaussian raw & 77.31\% & 87.50\% & 10.39\% & 20.13\% & 0.526 & $-53.44$\% \\
FF49 & Gaussian safety & 93.98\% & 91.67\% & 10.38\% & 20.83\% & 0.515 & $-55.13$\% \\
\addlinespace
ETFs & Equal weight & -- & -- & 5.44\% & 9.39\% & 0.470 & $-27.99$\% \\
ETFs & Empirical bootstrap & 0.00\% & 91.67\% & 5.29\% & 5.90\% & 0.679 & $-19.51$\% \\
ETFs & Predictive fixed & 82.87\% & 98.61\% & 5.08\% & 8.38\% & 0.475 & $-24.23$\% \\
ETFs & Gaussian raw & 80.56\% & 98.61\% & 5.06\% & 8.34\% & 0.474 & $-23.87$\% \\
ETFs & Gaussian safety & 95.37\% & 98.61\% & 5.46\% & 9.36\% & 0.473 & $-27.68$\% \\
\bottomrule
\end{tabular}
\end{table}

Safety adjustment raises the future-block proxy relative to raw and fixed radii by 16.67 and 14.81 percentage points for FF49 and by 14.81 and 12.50 points for ETFs.  These are descriptive paired-rebalance differences, not inferential dominance claims, and the resulting portfolios remain close to equal weight in risk and return.  Figure~\ref{fig:real-reliability} shows that violation rates vary across methods and universes, while Figure~\ref{fig:real-risk-return} places the same methods on the net risk--return plane with moving-block uncertainty.  The Pflug--Wozabal and Blanchet--Chen--Zhou comparators are retained in the replication tables as decision-performance references only \citep{pflug2007ambiguity,blanchet2022portfolio}: their objectives differ from the manuscript's mean--CVaR certificate, and the latter is exact only on the frozen capped support.

\begin{figure}[!tbp]
\centering
\includegraphics[width=\linewidth,height=.78\textheight,keepaspectratio]{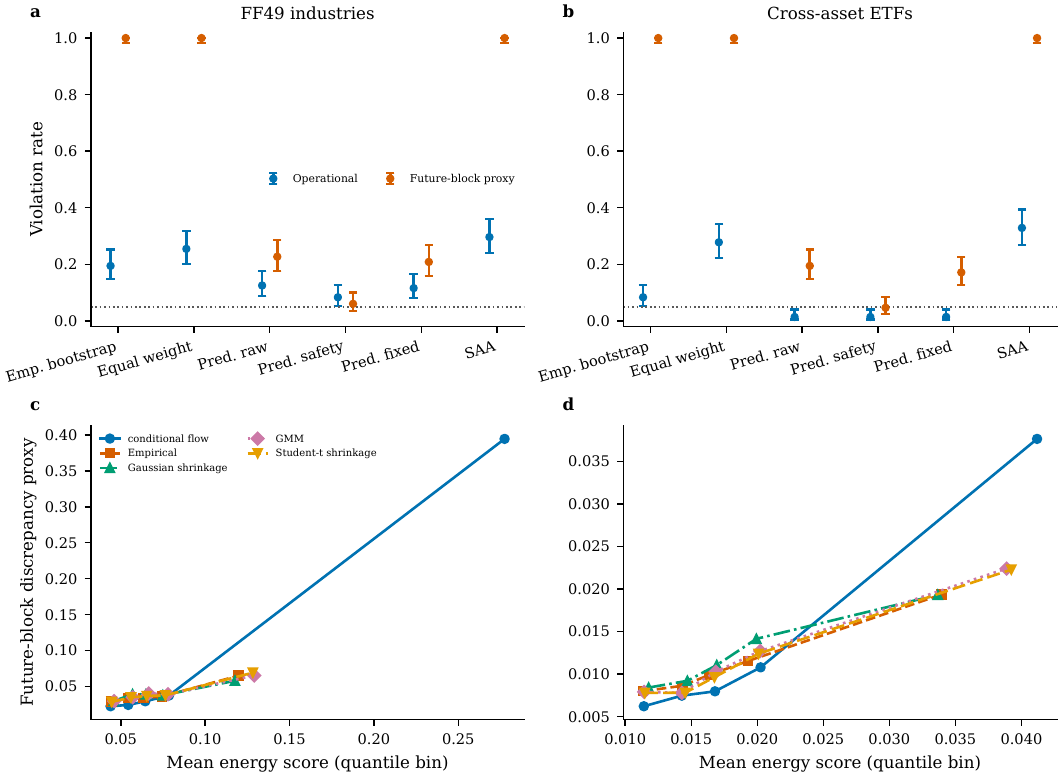}
\caption{Rolling real-data reliability diagnostics.  Panels (a)--(b) report operational and future-block violation rates by method for the FF49 and ETF universes; error bars are 95\% Wilson intervals and the dotted line is the 5\% target violation rate.  Panels (c)--(d) relate the future-block discrepancy proxy to mean energy score across quantile bins.  Future blocks are noisy realized proxies and do not identify coverage of the unknown contemporaneous return law.}
\label{fig:real-reliability}
\end{figure}

\begin{figure}[!tbp]
\centering
\includegraphics[width=\linewidth,height=.78\textheight,keepaspectratio]{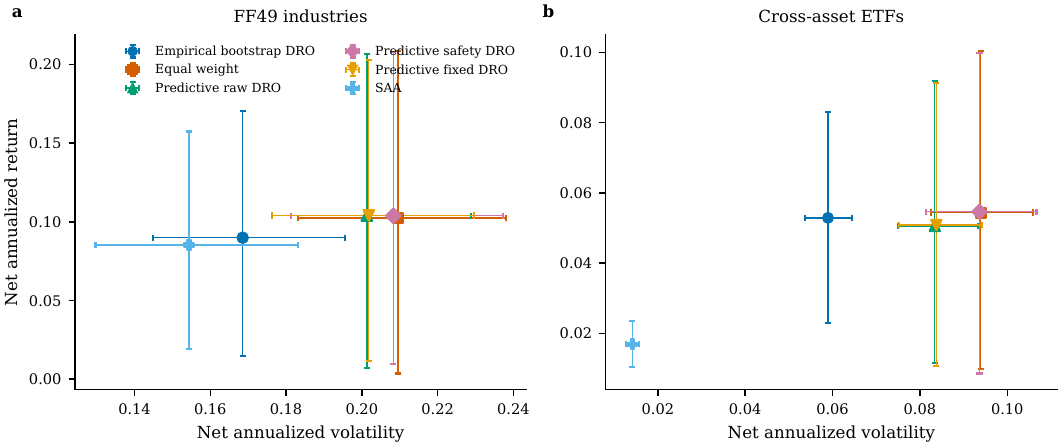}
\caption{Net risk--return comparison at the primary 10 basis-point transaction cost.  Points report annualized net return against annualized net volatility; horizontal and vertical error bars use the frozen 21-day moving-block bootstrap for serial returns.}
\label{fig:real-risk-return}
\end{figure}

Two limitations are visible only because the chronological design is enforced.  First, with 12 calibration pseudo-tasks the corrected 95\% conformal index is $\lceil13(0.95)\rceil=13$, so the finite-sample-correct radius is infinite.  Five corrected or weighted conformal variants are consequently inactive in every primary real-data cell rather than being reported as successful finite-radius methods.  Second, the Gaussian-safety strategy has negative annualized net returns in each frozen stress window: $-59.87$\%, $-61.86$\%, and $-15.33$\% for FF49 and $-25.74$\%, $-18.86$\%, and $-23.50$\% for ETFs during the global financial crisis, COVID crash, and inflation stress, respectively.  These short-window annualizations are descriptive, not tests of universal loss.

The raw-data audit found one reproducibility qualification.  The current TIP payload is 47 bytes smaller than the payload recorded in the acquisition manifest and has a different SHA-256 hash, although all 16 frozen processed CSV/Parquet files match their manifest and every real-result row carries the locked data signature.  Replacing TIP with returns parsed from the current raw cache changes the signature and changes a daily return by at most $2.79\times10^{-6}$.  The reported experiment is therefore reproducible from the intact frozen processed panel, but not byte-for-byte from the present raw cache alone.

\subsection{Ablations, computation, and synthesis}

The complete one-factor ablations vary center, radius rule, scenario count, calibration size, proper-training/calibration split, radius features, model capacity, ground norm, feature scaling, contamination, and transaction costs; they also include an uninformative negative-control predictor.  Figures~\ref{fig:ablation-main-appendix} and \ref{fig:ablation-extra-appendix} report these results, while Figures~\ref{fig:frontier-radius-appendix}--\ref{fig:frontier-regret-appendix} show the reliability frontiers.  The conclusions are stable across these checks: conformal is far less conservative than the fixed radius, but is not uniformly more efficient than raw or bootstrap rules once reliability and decision loss are considered jointly.

Computational diagnostics are reported in Figure~\ref{fig:computation-appendix}.  Every synthetic portfolio row is feasible at the stated tolerance.  The contextual stage records two failed attempted solves, and the Blanchet--Chen--Zhou SLSQP comparator fails in 120 of 6,411 attempts, with maximum residual $2.77\times10^{-4}$; these failures are retained rather than silently dropped.  Exact one-dimensional transport and entropic multivariate transport are reported under different names throughout.

Taken together, the six prespecified questions receive mixed answers.  Predictive-center accuracy improves in selected families but not uniformly.  Increasing $M$ reduces scenario error but leaves measurable undercoverage at $10N$.  Calibration succeeds for the pooled newsvendor estimand and the saved synthetic proxy, but fails under important contextual shifts and is unavailable at finite radius in the main chronological conformal backtest.  Learned radii are efficient relative to a deliberately conservative fixed rule, not uniformly relative to raw or bootstrap radii.  Most importantly, the experiments do not support a general decision-value claim: calibration changes reliability far more than regret or realized financial performance.  Adaptive sequential procedures recover after the abrupt shift, but severe misspecification and real stress remain substantive failure modes.

\FloatBarrier

\section{Limitations and Scope}

The finite-sample certificate transfer developed in this paper is deliberately modular.  If a random ambiguity set contains the data-generating law, then the corresponding robust value is a valid upper certificate; this implication does not, by itself, constitute a general learning theorem for the center or radius.  The substantive statistical problem is to construct a small ambiguity set satisfying $\Prob\{\dW(P,\Pbar)\le\epshat\}\ge1-\beta$.  The oracle target $R_N=\dW(P,\Pbar)$ is unavailable in an ordinary single-environment application.  Supervised radius learning therefore requires simulated tasks with known laws, repeated environments, or an observable proxy whose approximation error can be controlled.  Without one of these structures, predictive accuracy alone does not establish distributional validity.

The strength of the empirical evidence differs across experiments.  In the newsvendor setting, $W_1$ is evaluated numerically from quantiles and the experiment directly studies the intended containment event whenever the calculation is determinate, subject to uncertified grid error.  The multivariate portfolio and contextual experiments instead use a fixed-subsample, entropically regularized debiased-Sinkhorn discrepancy.  These quantities are useful diagnostics but are not exact Wasserstein distances.  In the synthetic portfolio audit, only 52.0\% of paired label calculations meet the strict saved convergence criterion.  In the 12-case exact-assignment sensitivity, 10 cases converge at the main regularization; their mean absolute error relative to exact assignment is 0.0122, or 25.4\% in relative terms.  We therefore make no exact multivariate distributional-coverage claim.  Finite scenario generation is also a separate approximation: in the fixed-radius newsvendor scenario diagnostic, numerical containment remains 94.225\%, below the 95\% target, at $M=10N$.

The conformal conclusions are assumption-specific.  Ordinary split conformal supplies marginal validity across exchangeable tasks; it does not imply conditional validity within every distribution family, volatility level, or dimension.  This distinction is visible empirically: pooled calibration can be close to nominal while important subgroups remain undercovered.  The contextual trajectory is nonexchangeable, so the weighted and adaptive procedures evaluated there have sequential, assumption-dependent interpretations rather than the exchangeable guarantee proved for split conformal.  In the rolling financial experiment, a future return block is only a noisy proxy for the contemporaneous law $P_t$.  Moreover, 12 calibration pseudo-tasks are insufficient to form a finite corrected 95\% conformal order statistic.  The resulting infinite radii and inactive conformal variants are a finite-sample limitation, not evidence of real-data coverage.

Calibration should not be conflated with decision improvement.  In the unrestricted-support $W_1$ newsvendor model, the radius contributes an additive Lipschitz penalty and does not alter the optimizer.  In the synthetic portfolio core, calibration changes proxy containment but yields no detectable regret improvement.  The conditional contextual model performs worse in aggregate than the unconditional empirical benchmark and fails sharply under predictive misspecification.  In the rolling portfolios, moving-block return intervals overlap across leading methods, and no stored paired return test supports performance dominance.  These experiments therefore establish neither universal radius efficiency nor blanket superiority of predictive DRO.

The financial study considers long-only, fully invested portfolios, monthly rebalancing, proportional transaction costs, two asset universes, and one historical period.  It omits market impact, taxes, borrowing constraints, short sales, and implementation-capacity effects.  The Pflug--Wozabal and Blanchet--Chen--Zhou comparators solve different objectives and thus provide decision-performance benchmarks rather than like-for-like comparisons of mean--CVaR certificates.  The reported results are tied to hash-locked processed inputs.  A provenance audit found that the current cached raw TIP response does not byte-match the acquisition hash, although every processed input and final experiment signature remains intact; hence the backtest is reproducible from the frozen processed panel, but the raw payload cannot presently be reconstructed byte-for-byte.

\section{Conclusion}

This paper develops a modular framework for Wasserstein DRO with a learned nominal distribution and a learned, data-dependent radius.  The theory shows that empirical centering is not essential: reliability follows from calibrated coverage of the actual center used by the optimizer.  It provides finite-sample and asymptotic guarantees for general random centers, tractable reformulations for non-uniform discrete laws, a decomposition of predictive-model and scenario-discretization errors, stability under joint center--radius perturbations, high-probability and expected-regret bounds, and pointwise optimality of the oracle conditional-quantile radius.  A finite-sample-correct split-conformal construction supplies marginal calibration across exchangeable tasks, while the distinction between distributional and certificate-targeted radii clarifies what each guarantee does and does not imply.

The experiments support a more qualified conclusion than uniform superiority of learned ambiguity sets.  In the one-dimensional benchmark, split conformal attains 95.506\% pooled numerical containment at the 95\% target, but the raw learned radius already attains 95.575\%, and both predictive rules have higher regret than empirical bootstrap.  In the synthetic mean--CVaR core, conformal calibration raises the saved Sinkhorn-proxy containment rate from 80.850\% to 95.275\% and is substantially less conservative than the fixed radius, yet it produces no detectable regret improvement.  Under contextual shift, the conditional model reaches only 70.325\% aggregate proxy containment compared with 88.800\% for the unconditional empirical benchmark.  After the abrupt shift, the adaptive controller and exponentially weighted rule recover after 123 and 292 observations, respectively, whereas the fixed exchangeable-split diagnostic has no recorded recovery.  In the real-data study, safety adjustment improves the future-block proxy in both universes and raises operational reliability for FF49, while tying the raw and fixed rules at 98.611\% for ETFs; return intervals overlap competing methods and do not support performance dominance.

The combined message is that calibration can support statistical reliability but does not guarantee decision value.  Center quality, scenario approximation, radius calibration, optimization structure, and distribution shift are separate links; improving one need not improve the final decision.  Learned ambiguity sets are most compelling when side information reduces center error, the radius target is observed or conservatively approximated, calibration assumptions match deployment, and decisions respond to lower conservatism.  Priorities for future work are scalable discrepancy estimators with controlled approximation error, conditional and sequential calibration under weaker dependence, and decision-aware radius learners with an explicit distributional interpretation.  This framework evaluates such advances without conflating predictive accuracy, ambiguity-set coverage, and downstream performance.

\appendix

\section{Additional Linear Reformulation for Polyhedral Support}

Suppose $\XiSet=\{\xi:C\xi\le d\}$ and
\begin{equation}
    \ell(\xi)=\max_{k\le K}\{\inner{a_k}{\xi}+b_k\}.
\end{equation}
Then Theorem \ref{thm:convex-reform-general} yields
\begin{equation}\label{eq:pwa-polytope}
\begin{aligned}
\Phi_\varepsilon(\ell;\Pbar)=
\inf_{\lambda,s_j,\gamma_{jk}}\quad
&\lambda\varepsilon+\sum_{j=1}^Mp_js_j\\
\mathrm{s.t.}\quad
&b_k+\inner{a_k}{\zeta_j}+\inner{\gamma_{jk}}{d-C\zeta_j}\le s_j,
\qquad \forall j,k,\\
&\norm{C^\top\gamma_{jk}-a_k}_*\le\lambda,
\qquad \forall j,k,\\
&\gamma_{jk}\ge0,
\qquad \forall j,k.
\end{aligned}
\end{equation}
If $\norm{\cdot}$ is the $\ell_1$ or $\ell_\infty$ norm, this becomes a linear program.

\section{Algorithmic Templates}

\begin{algorithm}[ht]
\caption{Predictive-center Wasserstein DRO}
\begin{algorithmic}[1]
\Require Training data $\mathcal D_N$, confidence level $1-\beta$, predictive distribution model $G_\theta$, radius rule $\varepsilon_N$.
\State Fit $G_\theta$ on $\mathcal D_N$.
\State Generate predictive distribution $\Ptil=\sum_{j=1}^Mp_j\delta_{\zeta_j}$.
\State Compute or predict radius $\varepsilon_N$.
\State Solve $\inf_{x\in X}\sup_{Q\in B_{\varepsilon_N}(\Ptil)}\E_Q[h(x,\xi)]$ using the dual reformulation.
\State Return decision $x_N$, certificate $J_N$, center error proxy, and radius.
\end{algorithmic}
\end{algorithm}

\begin{algorithm}[ht]
\caption{Conformal learned radius across exchangeable tasks}
\begin{algorithmic}[1]
\Require Proper-training tasks for fitting $f_\theta$; separate calibration tasks $t=1,\ldots,T$ with features $Z_t$ and target radii $R_t$; new feature $Z_{\mathrm{new}}$; confidence level $1-\beta$.
\State Fit $f_\theta$ using only the proper-training tasks.
\State Compute calibration residuals $S_t=R_t-f_\theta(Z_t)$.
\State Set $k=\lceil(T+1)(1-\beta)\rceil$ and let $q_{T,\beta}=S_{(k)}$ if $k\le T$ (otherwise $q_{T,\beta}=+\infty$).
\State Set $\widehat\varepsilon_{\mathrm{new}}=\max\{0,f_\theta(Z_{\mathrm{new}})+q_{T,\beta}\}$.
\State Solve DRO with radius $\widehat\varepsilon_{\mathrm{new}}$.
\end{algorithmic}
\end{algorithm}

\clearpage
\section{Additional Experimental Exhibits}\label{app:experimental-exhibits}

This appendix collects the diagnostic, real-data path, ablation, and computational exhibits omitted from the main text.  All figures use the frozen protocol and inferential conventions of Section~\ref{sec:experiments}.  Multivariate containment remains an entropically regularized discrepancy proxy, and future-block financial quantities remain realized proxies rather than distributional-coverage estimates.  Complete task-level outputs and publication tables remain in the replication archive.

\begin{figure}[!htbp]
\centering
\includegraphics[width=\linewidth,height=.78\textheight,keepaspectratio]{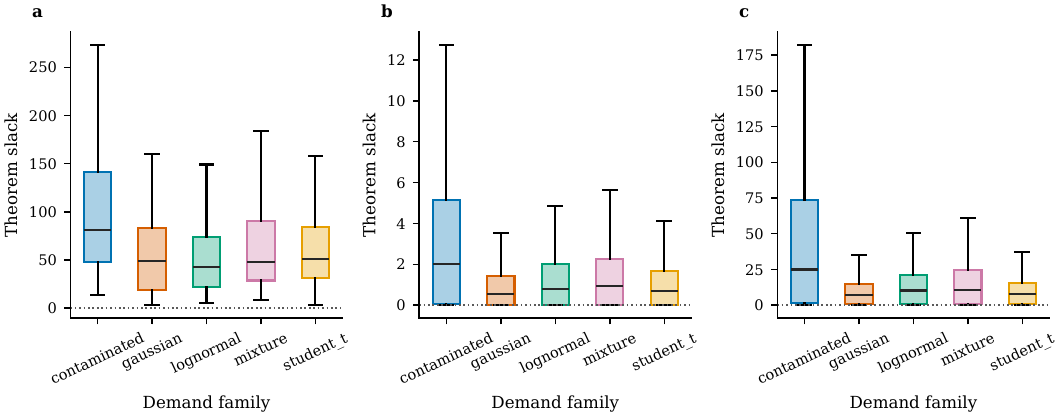}
\caption{Numerical verification of the newsvendor bounds.  Panels (a)--(c) summarize, by demand family, the slack in the normalized-regret bound, the scenario-error triangle inequality, and the center--radius stability bound, respectively.  Nonnegative slack indicates satisfaction.  All 1,714,585 recorded checks pass at the frozen numerical tolerance.}
\label{fig:theorem-checks-appendix}
\end{figure}

\begin{figure}[p]
\centering
\includegraphics[width=\linewidth,height=.78\textheight,keepaspectratio]{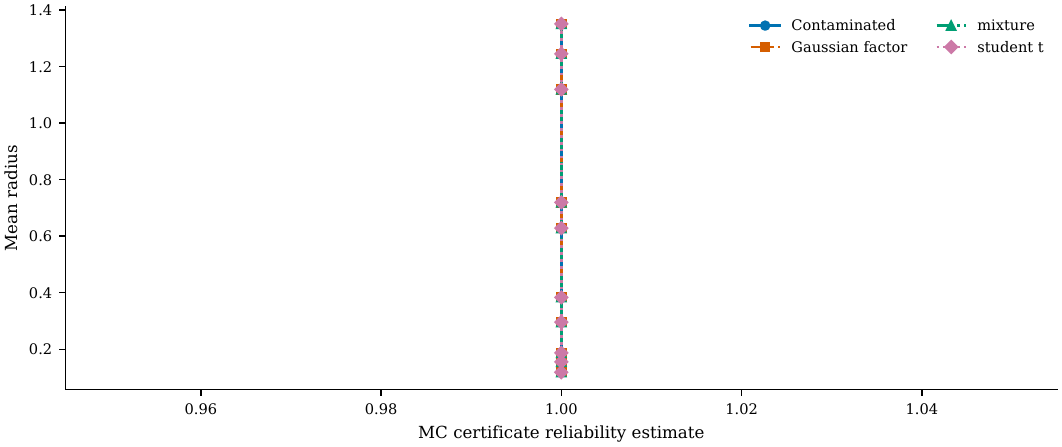}
\caption{Reliability--radius frontier in the paired synthetic mean--CVaR experiment.  Reliability is the Monte Carlo certificate-reliability estimate, and radius is averaged over paired tasks.  All core configurations attain an estimated reliability of one, so the figure distinguishes conservatism but provides no statistical separation in certificate reliability.}
\label{fig:frontier-radius-appendix}
\end{figure}

\begin{figure}[p]
\centering
\includegraphics[width=\linewidth,height=.78\textheight,keepaspectratio]{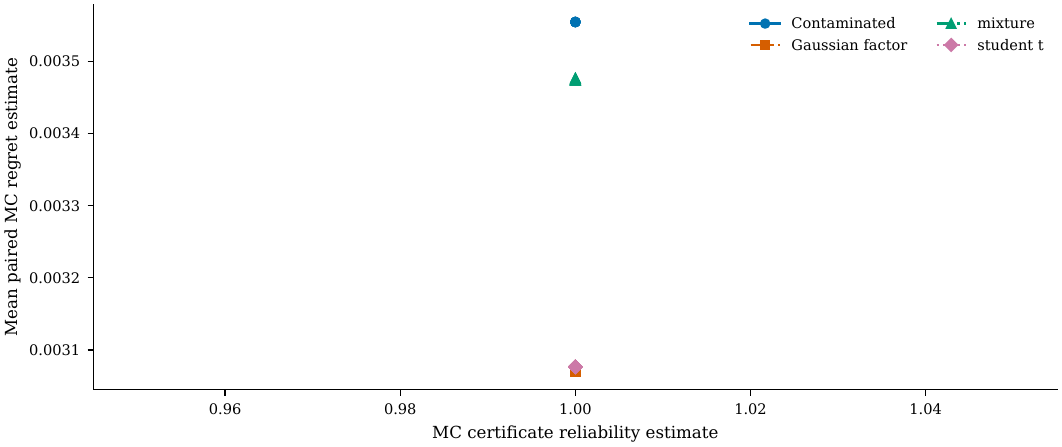}
\caption{Reliability--regret frontier in the paired synthetic mean--CVaR experiment.  Regret is evaluated on the common paired Monte Carlo sample.  The near-coincident points show that large differences in discrepancy-proxy containment do not translate into detectable differences in core decision loss.}
\label{fig:frontier-regret-appendix}
\end{figure}

\begin{figure}[p]
\centering
\includegraphics[width=\linewidth,height=.78\textheight,keepaspectratio]{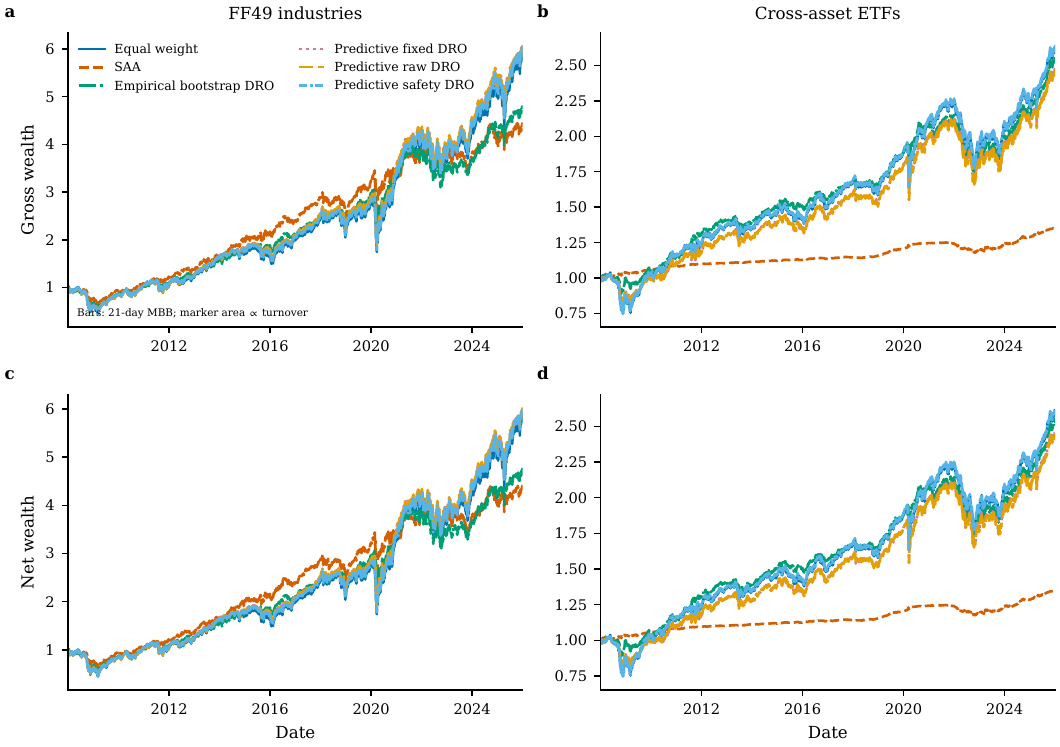}
\caption{Cumulative wealth over the common rolling evaluation period.  Panels (a) and (b) report gross wealth for the FF49 and cross-asset ETF universes; panels (c) and (d) report the corresponding net wealth after the frozen proportional transaction cost.  Each series uses the next-calendar-month holdings of the main chronological backtest.}
\label{fig:wealth-appendix}
\end{figure}

\begin{figure}[p]
\centering
\includegraphics[width=\linewidth,height=.78\textheight,keepaspectratio]{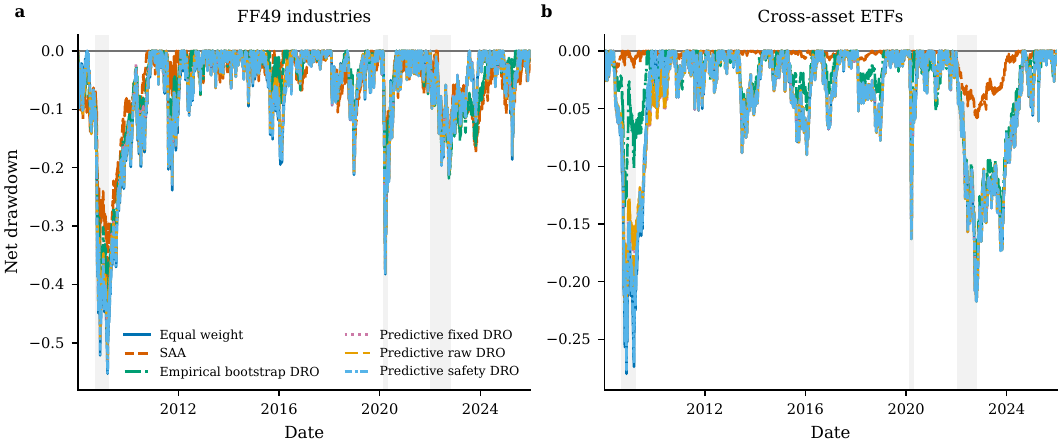}
\caption{Net drawdown paths for (a) the FF49 universe and (b) the cross-asset ETF universe.  Shading marks the prespecified global-financial-crisis, COVID-crash, and inflation-stress windows.  These paths are descriptive historical outcomes, not independent stress-test replications.}
\label{fig:drawdown-appendix}
\end{figure}

\begin{figure}[p]
\centering
\includegraphics[width=\linewidth,height=.78\textheight,keepaspectratio]{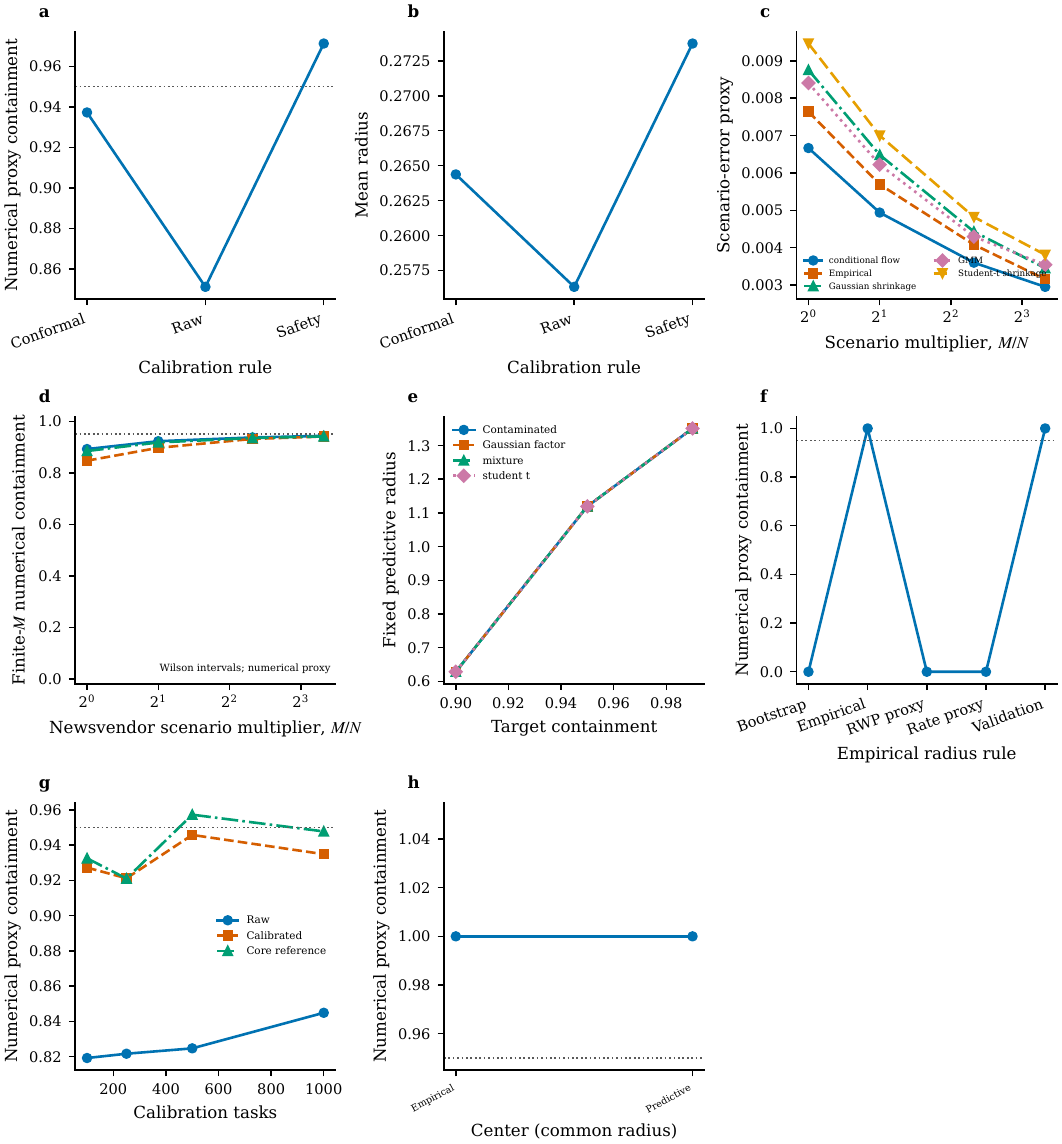}
\caption{Primary ablations.  Panels (a)--(b) vary the calibration rule; (c)--(d) vary the scenario multiplier; (e) varies the target used to construct the fixed predictive radius; (f) compares empirical-radius rules; (g) varies the calibration-task count; and (h) compares empirical and predictive centers at a common radius.  Every containment ordinate is explicitly a numerical proxy.}
\label{fig:ablation-main-appendix}
\end{figure}

\begin{figure}[p]
\centering
\includegraphics[width=\linewidth,height=.78\textheight,keepaspectratio]{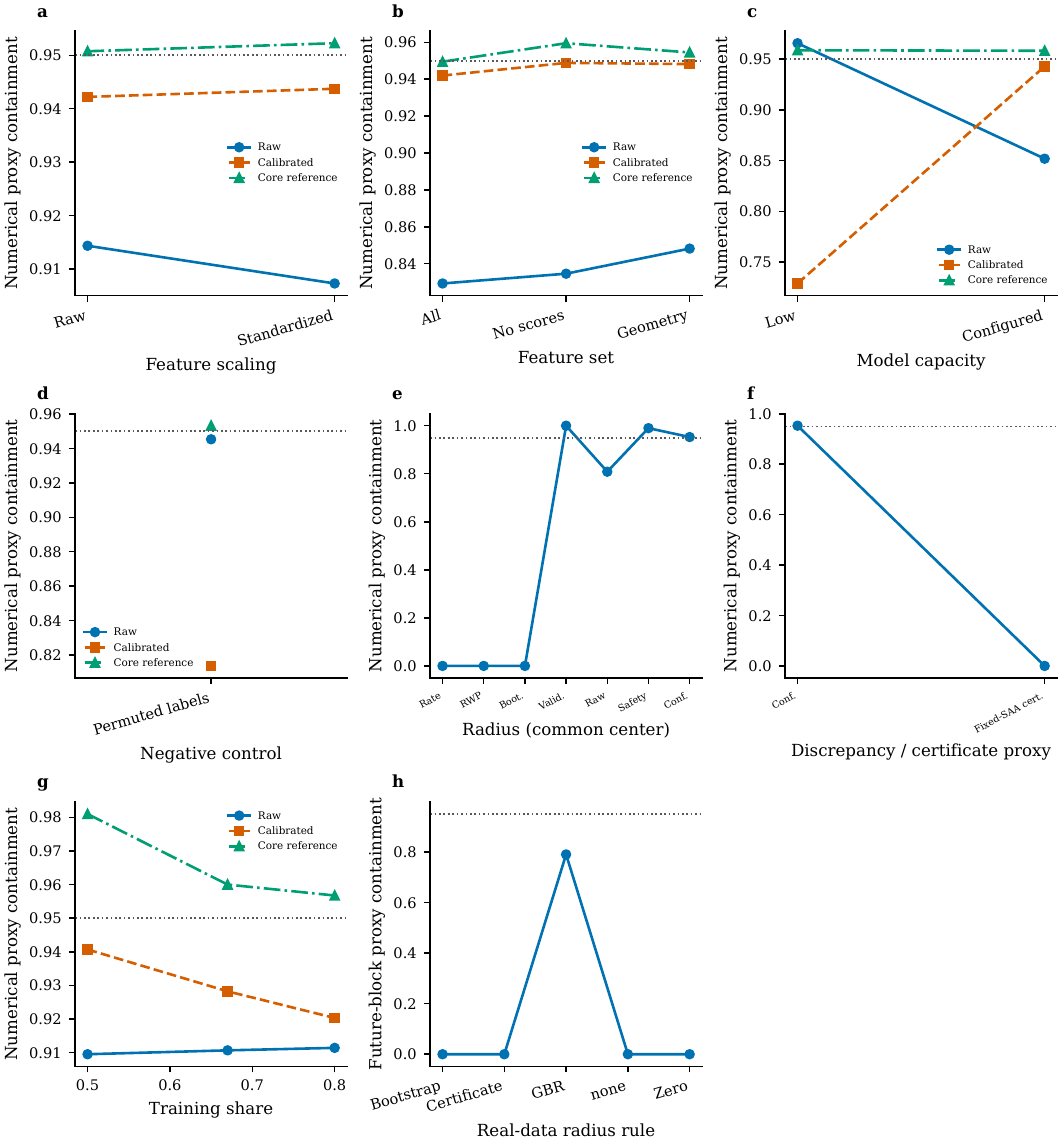}
\caption{Additional ablations.  Panels (a)--(c) vary feature scaling, feature selection, and model capacity; (d) permutes the labels as a negative control; (e) compares radius rules at a common center; (f) contrasts discrepancy and certificate proxies; (g) varies the proper-training share; and (h) compares real-data radius rules using the future-block proxy.  Complete numerical rows are retained in the replication archive.}
\label{fig:ablation-extra-appendix}
\end{figure}

\begin{figure}[p]
\centering
\includegraphics[width=\linewidth,height=.78\textheight,keepaspectratio]{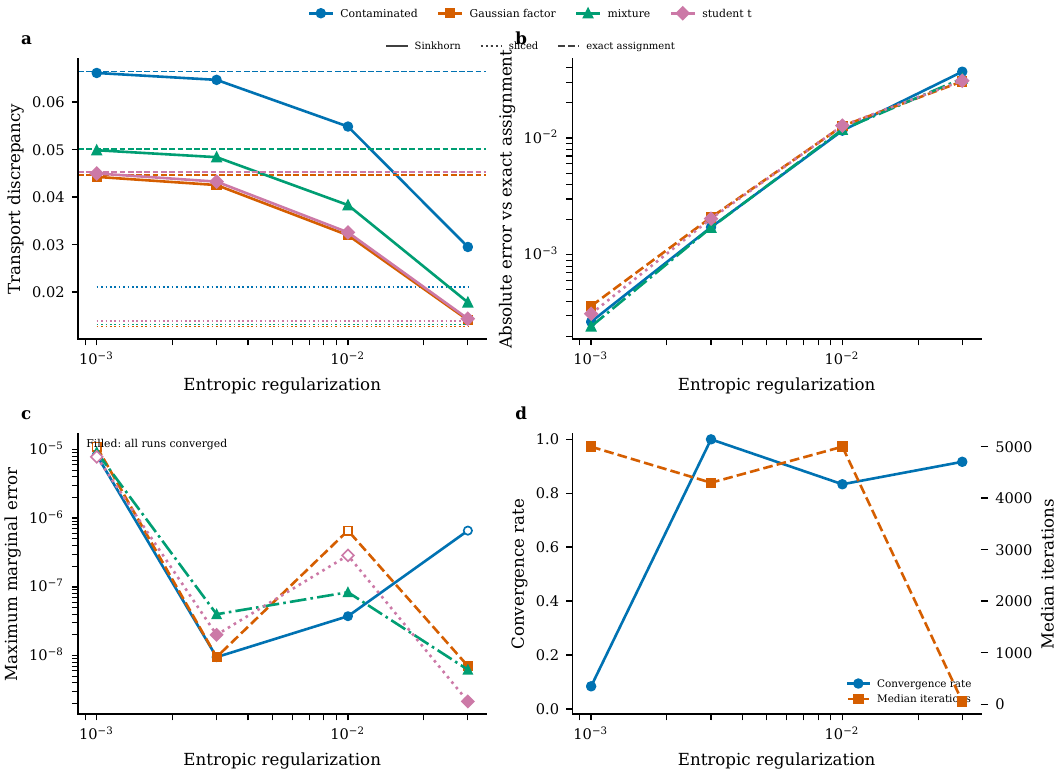}
\caption{Sensitivity of the multivariate discrepancy proxy to entropic regularization.  Panel (a) compares the debiased-Sinkhorn value with sliced-$W_1$ and exact-assignment references; (b) reports absolute error relative to exact assignment; (c) reports the maximum marginal error, with open markers denoting cases in which not all runs converged; and (d) reports convergence rates and median iteration counts.  The exact-assignment discrepancy demonstrates why multivariate containment is interpreted as a numerical proxy rather than exact Wasserstein coverage.}
\label{fig:sinkhorn-appendix}
\end{figure}

\begin{figure}[p]
\centering
\includegraphics[width=\linewidth,height=.78\textheight,keepaspectratio]{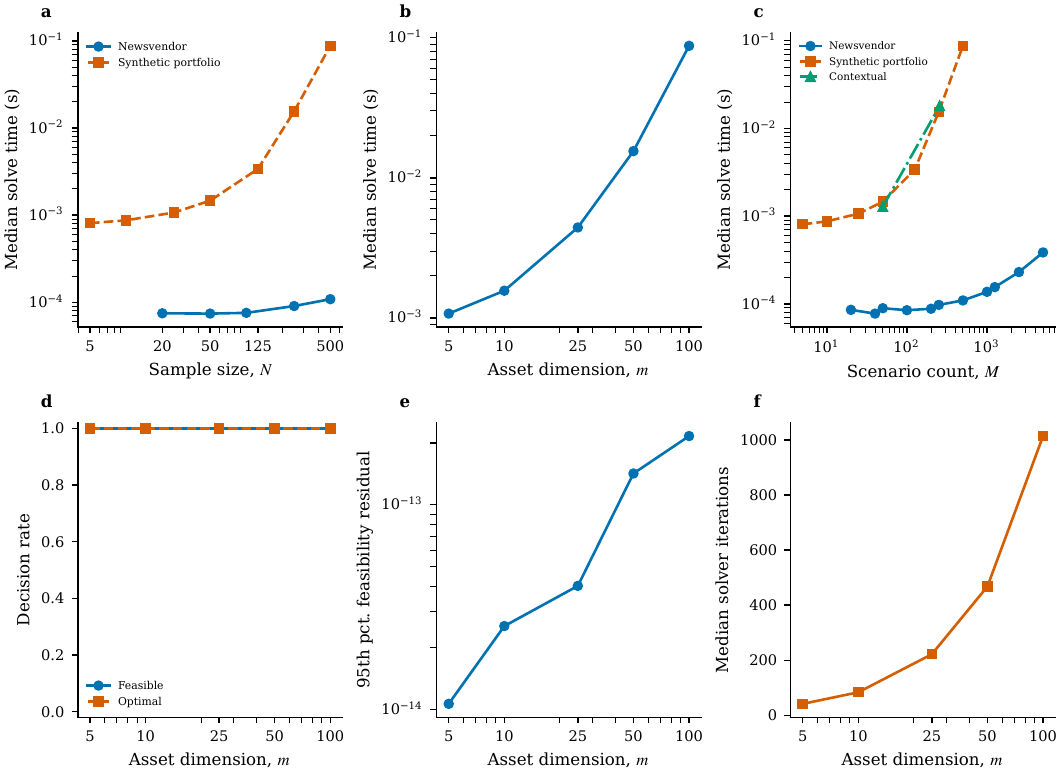}
\caption{Computational scaling and numerical diagnostics.  Panels (a)--(c) report median solve time against sample size, asset dimension, and scenario count; (d) reports feasible and optimal decision rates; (e) reports the 95th percentile feasibility residual; and (f) reports median solver iterations.  Failed attempted solves remain in the denominators of the relevant rates.}
\label{fig:computation-appendix}
\end{figure}

\FloatBarrier
\begingroup
\setlength{\emergencystretch}{1em}

\endgroup

\end{document}